\documentclass[12pt]{article}
\usepackage{graphicx}
\usepackage{amsmath,amssymb,amsthm,amsfonts}
\newtheorem{thm}{Theorem}[section]
\newtheorem{cor}[thm]{Corollary}
\newtheorem{lem}[thm]{Lemma}

\newcommand{\R}{{\mathbb{R}}}
\newcommand{\Z}{{\mathbb{Z}}}

\newcommand{\vp}{\varphi}

\newcommand{\La}{\triangle}

\newcommand{\1}{\partial}

\newcommand{\3}{\varepsilon}
\newcommand{\4}{\widetilde}

\def\ni{\noindent}

\begin{document}
\title{Convergence of the Dirichlet solutions\\ 
of the very fast diffusion equation}
\author{Kin Ming Hui and Sunghoon Kim\\ 
Institute of Mathematics, Academia Sinica,\\
Taipei, Taiwan, R.O.C.\\
kmhui@gate.sinica.edu.tw, gauss79@math.sinica.edu.tw}
\date{Dec 15, 2010}
\smallbreak \maketitle
\begin{abstract}
For any $-1<m<0$, $\mu>0$, $0\le u_0\in L^{\infty}(\R)$ 
such that $u_0(x)\le (\mu_0 |m||x|)^{\frac{1}{m}}$ for any $|x|\ge R_0$
and some constants $R_0>1$ and $0<\mu_0\leq \mu$, and $f,\,g \in C([0,\infty))$ such that
$f(t),\, g(t) \geq \mu_0$ on $[0,\infty)$ we prove that as $R\to\infty$ the solution $u^R$ of the 
Dirichlet problem $u_t=(u^m/m)_{xx}$ in $(-R,R)\times (0,\infty)$, 
$u( R,t)=(f(t)|m|R)^{1/m}$, $u(-R,t)=(g(t)|m|R)^{1/m}$ for all $t>0$, $u(x,0)=u_0(x)$ in 
$(-R,R)$, converges uniformly on every compact subsets of $\R\times
(0,T)$ to the solution of the equation $u_t=(u^m/m)_{xx}$ in
$\R\times (0,\infty)$, $u(x,0)=u_0(x)$ in $\R$, which satisfies
$\int_{\R}u(x,t)\,dx=\int_{\R}u_0dx-\int_0^t(f(s)+g(s))\,ds$ for all $0<t<T$ 
where $\int_0^T(f+g)\,ds=\int_{\R}u_0dx$. We also prove that 
the solution constructed is equal
to the solution constructed in \cite{Hu3} using approximation by
solutions of the corresponding Neumann problem in bounded cylindrical
domains. 
\end{abstract}

\vskip 0.2truein


\vskip 0.2truein
\setcounter{equation}{0}
\setcounter{section}{-1}

\section{Introduction}
\setcounter{equation}{0}
\setcounter{thm}{0}

The equation
\begin{equation}\label{eq-base-phi-expression}
u_t=\Delta\phi_m(u)
\end{equation}
where $\phi_m(u)=u^m/m$ for $m\ne 0$ and $\phi_m(u)=\log u$ for $m=0$
arises in many physical models. For example when $m=0$ and the dimension
$n=1$ P.L.~Lions and G.~Toscani \cite{LT} and T.~Kurtz \cite{Ku} have 
shown that \eqref{eq-base-phi-expression} arises as the diffusive limit for finite velocity Boltzmann
kinetic models. When $m=0$ and $n=2$, the equation arises in the Ricci
flow on the complete manifold $\R^2$ \cite{DDD}, \cite{DD}, \cite{DP},
\cite{Hs1}, \cite{W1}, \cite{W2}. When $m=-1$, the equation appears in 
the model of heat conduction in solid hydrogen \cite{R}. 

When $m=1$, the equation
is the well known heat equation. When $0<m<1$, the equation is called the
fast diffusion equation and when $m<0$ the equation is called the very 
fast diffusion equation. We refer the reader to the survey papers of 
Aronson \cite{A} and Peletier \cite{P} and the book \cite{V2} by 
J.L.~Vazquez for various results on \eqref{eq-base-phi-expression}. 

As observed by J.L.~Vazquez \cite{V1} the behaviour of the solution 
of \eqref{eq-base-phi-expression} for $m\le 0$ is very different from the behaviour of solution
of \eqref{eq-base-phi-expression} for $m>0$. For example when $m>0$ and $n=1$ the solution
of \eqref{eq-base-phi-expression} preserves the mass while for $-1<m\le 0$ and $n=1$ there
exists finite mass solutions which vanish in a finite time \cite{RV}.    
In \cite{RV} A.~Rodriguez and J.L.~Vazquez by using semigroup method 
proved that for any $-1<m\le 0$, $0\le u_0\in L^1(\R)$ and any 
non-negative functions $f,g\in L_{loc}^{\infty}(0,\infty)$ there exists 
a smooth unique solution $u$ for 
\begin{equation}\label{eq-Cauchy-problem}
\left\{\begin{aligned}
&u_t=(u^{m-1}u_x)_x\quad\mbox{ in }\R\times (0,T)\\
&u(x,0)=u_0(x)\quad\mbox{ in }\R
\end{aligned}\right.
\end{equation}
which satisfies
\begin{equation}\label{eq-mass-1}
\int_{\R}u(x,t)\,dx=\int_{\R}u_0\,dx-\int_0^t(f+g)\,ds
\quad\forall 0\le t<T
\end{equation}
and
\begin{equation}\label{eq-property-infty-1}
\lim_{x\to\infty}\int_{t_1}^{t_2}u^{m-1}u_x(x,s)\,ds
=-\int_{t_1}^{t_2}f(s)\,ds
\quad\forall 0<t_1<t_2<T
\end{equation}
and
\begin{equation}\label{eq-property-infty-2}
\lim_{x\to -\infty}\int_{t_1}^{t_2}u^{m-1}u_x(x,s)\,ds
=\int_{t_1}^{t_2}g(s)\,ds
\quad\forall 0<t_1<t_2<T
\end{equation}
where 
\begin{equation}\label{eq-time-1}
T=\sup\biggl \{t'>0:\int_{\Bbb{R}}u_0\,dx>\int_0^{t'}(f+g)\,ds\biggr \}.
\end{equation}
Later K.M.~Hui \cite{Hu3} give another proof of this result by proving 
that the solutions of the Neumann problem
\begin{equation*}
\left\{\begin{aligned}
&u_t=\Delta \phi_m (u),u>0,\quad\mbox{ in }(-R,R)\times (0,T)\\
&(\phi_m(u))_x(-R,t)=-f(t)\quad\forall 0<t<T\\
&(\phi_m(u))_x(-R,t)=g(t)\quad\forall 0<t<T\\
&u(x,0)=u_{0}(x)\qquad\qquad\mbox{in }(-R,R)
\end{aligned}\right.
\end{equation*}
converges uniformly on every compact subset of $\R\times (0,T)$
to the solution of \eqref{eq-Cauchy-problem} which satisfies 
\eqref{eq-mass-1}, \eqref{eq-property-infty-1}, 
\eqref{eq-property-infty-2} and \eqref{eq-time-1} 
as $R\to\infty$.

In this paper we will prove that for any $-1<m<0$, $\mu>0$, $0\leq u_0\in L^{\infty}(\R)$ which satisfies \eqref{eq-condition-u_0} as $R\to\infty$
the solution $u^R$ of the Dirichlet problem 
\begin{equation}\label{problem-Dirichlet}
\begin{cases}
\begin{aligned}
&u_t=\left(u^{m}/m\right)_{xx}\qquad\qquad
\mbox{in $(-R,R)\times[0,\infty)$}\\
&u(\pm R,t)=(\mu|m|R)^{\frac{1}{m}}\quad\forall t>0\\
&u(x,0)=u_0(x)\qquad\qquad\mbox{ in }(-R,R) 
\end{aligned}
\end{cases}
\end{equation}
converges uniformly on every compact subsets of $\R\times
(0,T)$ to a solution $u$ of \eqref{eq-Cauchy-problem} which satisfies 
\begin{equation}\label{eq-thm-condition-1}
\int_{\R}u(x,t)\,dx=\int_{\R}u_0\,dx-2\mu t\quad\forall
0<t<T
\end{equation}
and 
\begin{equation}\label{eq-condition-neumann-2}
\frac{u^{m}(x,t)}{m|x|}\to -\mu\quad\mbox{ uniformly in }[a,b]
\quad\mbox{ as }|x|\to\infty
\end{equation}
for any $0<a<b<T$ where 
\begin{equation}\label{eq-mu-time}
T=\frac{1}{2\mu}\int_{\R}u_0\,dx.
\end{equation}
As a consequence by an approximation argument for any $0\le f\in 
L_{loc}^{\infty}([0,\infty))$ we obtain another proof of the existence
of solution of \eqref{eq-Cauchy-problem} which satisfies 
\begin{equation}\label{main-integral-condition-1}
\int_{\R}u(x,t)\,dx=\int_{\R}u_0\,dx-2\int_0^{t}f\,ds \qquad 
\forall 0\leq t<T,
\end{equation}
and
\begin{equation}\label{eq-limit-condition-infty}
\frac{u^{m}(x,t)}{m|x|}\to -f(t)\quad\mbox{ uniformly in }[a,b]
\quad\mbox{ as }|x|\to\infty
\end{equation}
for any $0<a<b<T$ where $T$ is given by 
\begin{equation}\label{eq-def-T-1}
2\int_0^Tf\,ds=\int_{\R} u_0\,dx.
\end{equation}
For any $f,\, g\in C([0,\infty))$ such that $f(t),\, g(t)\ge\mu_0$ on $[0,\infty)$
for some constant $\mu_0>0$ and $0\leq u_0\in L^{\infty}(\R)$ which satisfies \eqref{eq-condition-u_0} we also prove that the solution of 
\begin{equation}\label{problem-general-Dirichlet}
\begin{cases}
\begin{aligned}
&u_t=\left(u^{m}/m\right)_{xx}\qquad\qquad\quad
\mbox{in }(-R,R)\times[0,\infty)\\
&u(R,t)=(f(t)|m|R)^{\frac{1}{m}}\quad\forall t>0\\
&u(-R,t)=(g(t)|m|R)^{\frac{1}{m}}\quad\forall t>0\\
&u(x,0)=u_0(x)\qquad\qquad\quad\mbox{in }(-R,R). 
\end{aligned}
\end{cases}
\end{equation}
converges uniformly on every compact subsets of $\R\times
(0,T)$ to a solution $u$ of \eqref{eq-Cauchy-problem}
which satisfies 
\begin{equation}\label{main-integral-condition-general}
\int_{\R}u(x,t)\,dx=\int_{\R}u_0\,dx-\int_0^{t}(f+g)\,ds \qquad 
\forall 0\leq t<T,
\end{equation}
and 
\begin{equation}\label{eq-limit-condition-infty+general}
\frac{u^{m}(x,t)}{mx}\to -f(t)\quad\mbox{ uniformly in }[a,b]
\quad\mbox{ as }x\to\infty
\end{equation}
and 
\begin{equation}\label{eq-limit-condition-infty-general}
\frac{u^{m}(x,t)}{mx}\to g(t)\quad\mbox{ uniformly in }[a,b]
\quad\mbox{ as }x\to-\infty
\end{equation}
for any $0<a<b<T$ where $T$ is given by 
\begin{equation}\label{eq-def-T-general}
\int_0^T(f+g)\,ds=\int_{\R} u_0\,dx.
\end{equation}
 as $R\to\infty$.

A natural question to ask is that whether the solution $u$ of 
\eqref{eq-Cauchy-problem} which satisfies either \eqref{main-integral-condition-1} or \eqref{main-integral-condition-general}
for some function $f$, $g$ 
constructed by the Dirichlet approximation is equal to the solution
of \eqref{eq-Cauchy-problem} that also satisfies either \eqref{main-integral-condition-1} or \eqref{main-integral-condition-general} constructed in \cite{Hu3} by Neumann 
approximation. In this paper we answer this question in the affirmative 
and prove that the two solutions are equal.

The plan of this paper is as follows. In section one we will 
construct a symmetric self-similar solution of (0.1). In section two 
we will prove some properties of the Green function for the Laplace 
operator $\Delta$ in $(-R,R)$ for any $R>1$. In section three we 
will prove the convergence results of the Dirichlet solutions of 
\eqref{problem-Dirichlet} as $R\to\infty$. In section four we will 
prove the equality of the solutions of \eqref{eq-Cauchy-problem} 
constructed by the Dirichlet approximation and by the Neumann approximation. 
We will also prove the convergence of solutions of 
\eqref{problem-general-Dirichlet} as $R\to\infty$ in section four.   

We start will some definitions. We will assume that $-1<m<0$ for the
rest of the paper. For any $R>0$ and $T>0$ let $I_R=(-R,R)$, 
and $Q_R^T=I_R\times (0,T)$. For any 
$0\le u_0\in L^{\infty}(I_R)$ and $g\in L^{\infty}(\{\pm R\}\times[0,T))$
such that $\inf_{\{\pm R\}\times[0,T)}g>0$, we say that $u$ is a solution of 
the Dirichlet problem 
\begin{equation*}
\begin{cases}
\begin{aligned}
u_t=&(u^m/m)_{xx}\quad\mbox{on $ I_R\times (0,T)$}\\
u(\pm R,t)=&g(\pm R,t)\quad\mbox{ on $(0,\infty)$}\\
u(x,0)=&u_0(x)\qquad\,\,\mbox{ in }I_R
\end{aligned}
\end{cases}
\end{equation*}
if $0<u\in C^2(Q_R^T)\cap L^{\infty}(Q_R)$ satisfies
\begin{equation}\label{main-very-fast-diffusion-1}
u_t=(u^m/m)_{xx} 
\end{equation}
in $Q_R^T$ with
\begin{equation*}
\int_{t_1}^{t_2}\int_{I_R}\left[\left(\frac{u^m}{m}\right)\psi_{xx}
+u\psi_t\right]\,dxds=\int_{t_1}^{t_2}\int_{\partial I_R}
\left(\frac{g^m}{m}\right)\frac{\1\psi}{\1\nu}\,d\sigma ds
+\int_{I_R} u\psi\,dx\Bigg|_{t_1}^{t_2}
\end{equation*}
for all $0<t_1<t_2<T$, $\psi\in C^{\infty}(I_R\times(0,T))$
such that $\psi\equiv 0$  on $\{\pm R\}\times (0,T)$,
where $\1/\1\nu$ is derivative with respect to the unit outward normal 
$\nu$ on $\{\pm R\}\times (0,T)$
and $u(\cdot, t)\to u_0$ in $L^1(-R,R)$ as $t\to 0$.
For any $0\le u_0\in L^1(\R)$ we say that $u$ is a solution of 
\eqref{eq-Cauchy-problem} in
$\R\times (0,T)$ if $u>0$ is a classical solution of 
\eqref{main-very-fast-diffusion-1} in $\R\times 
(0,T)$ and
$$
u(\cdot, t)\to u_0\quad\mbox{ in }L^1(\R)\quad\mbox{ as }t\to 0.
$$
For any set $A$ we let $\chi_A$ be the characteristic function of the
set $A$.

\section{Self-similar solutions of the very fast diffusion equation}
\setcounter{equation}{0}
\setcounter{thm}{0}

In this section we will use a modification of the technique of \cite{Hs3}
to construct self-similar solutions of \eqref{main-very-fast-diffusion-1}. 

\begin{lem}\label{lem-property-f}
For any $R_0>0$ and $\eta>0$, let $f(r)$ be the solution of 
\begin{equation}\label{main-symmetric-1}
\left(\frac{f'}{f^{1-m}}\right)'+\frac{1}{1+m}f-\frac{m}{1+m}rf'=0, \qquad f>0,
\end{equation}
in $(0,R_0)$ which satisfies
\begin{equation}\label{condition-boundary-1}
\begin{cases}
f(0)=\eta\\
f'(0)=0.
\end{cases}
\end{equation}
Then
\begin{equation}\label{property-f-1}
f-mrf'>0 \qquad \mbox{in $[0,R_0)$}
\end{equation}
and 
\begin{equation}\label{property-f-2}
f'<0 \qquad \mbox{in $(0,R_0)$}.
\end{equation}
\end{lem}

\begin{proof}
Let $h=f-mrf'$. By direct computation,
\begin{equation*}
h'+\left((m-1)\frac{f'}{f}-\frac{m}{1+m}rf^{1-m}\right)h=0
\quad\mbox{ in }(0,R_0).
\end{equation*}
Hence
\begin{equation*}
(g(r)h(r))'=0 \qquad \qquad \mbox{in $(0,R_0)$}
\end{equation*}
where
\begin{equation*}
\begin{aligned}
g(r)&=exp\left(-\frac{m}{1+m}\int_0^r\rho f(\rho)^{1-m}\,d\rho
+(m-1)\int_0^r(\ln f)'\,d\rho\right)\\
&=\left(\frac{f(0)}{f(r)}\right)^{1-m}exp\left(-\frac{m}{1+m}
\int_0^r\rho f(\rho)^{1-m}\,d\rho\right).
\end{aligned}
\end{equation*}
Since $h(0)=\eta>0$ and $g(r)>0$ in $(0,R_0)$, \eqref{property-f-1} 
follows. In addition, by \eqref{main-symmetric-1}, 
\eqref{condition-boundary-1} and \eqref{property-f-1},
\begin{equation*}
\left(\frac{f'}{f^{1-m}}\right)'=-\frac{1}{1+m}h<0\quad 
\mbox{in $(0,R_0)$}.
\end{equation*}
Hence
\begin{equation*}
\frac{f'(r)}{f^{1-m}(r)}<0 \qquad \qquad \mbox{in $(0,R_0)$}
\end{equation*}
and the lemma follows.
\end{proof}

\begin{lem}\label{lem-existence-symmetric}
For any $\eta>0$ there exists a unique solution $f$ of 
\eqref{main-symmetric-1} in $(0,\infty)$ which satisfies 
\eqref{condition-boundary-1}.
\end{lem}

\begin{proof}
Uniqueness of the solution of \eqref{main-symmetric-1}, 
\eqref{condition-boundary-1}, in $(0,\infty)$ follows by standard 
O.D.E. theory. So we only need to prove the existence of solution of 
\eqref{main-symmetric-1}, \eqref{condition-boundary-1}, in $(0,\infty)$. 
Local existence of solution of \eqref{main-symmetric-1}, 
\eqref{condition-boundary-1}, in a small interval around the origin 
also follows from standard O.D.E. theory.\\
\indent Let $(0,R_0)$ be the maximal interval of existence for the 
solution $f$ of \eqref{main-symmetric-1}, \eqref{condition-boundary-1}. 
Suppose $R_0<\infty$. Then there exists an increasing sequence 
$\{r_i\}_{i=1}^{\infty}$, $r_i\to R_o$ as $i \to \infty$, such that either
\begin{equation*}
|f'(r_i)|\to \infty \quad \mbox{as} \,\,i \to \infty
\end{equation*}
or
\begin{equation*}
f(r_i) \to 0 \quad \mbox{as}\,\,i \to \infty
\end{equation*}
or
\begin{equation*}
f(r_i) \to \infty \quad \mbox{as} \,\, i \to \infty. 
\end{equation*} 
By Lemma \ref{lem-property-f} \eqref{property-f-2} holds. Hence
\begin{equation}\label{decreasing-f}
0<f(r)\leq f(0) \qquad \forall 0\leq r< R_0.
\end{equation}
By integrating \eqref{main-symmetric-1}, we have
\begin{equation}\label{main-eq-integral-1}
\begin{aligned}
\frac{f'}{f^{1-m}}&=-\frac{1}{1+m}\int_0^{r}f(\rho)\,d\rho
+\frac{m}{1+m}\int_0^{r}\rho f'(\rho)\,d\rho\\
&=\frac{m}{1+m}rf(r)-\int_0^{r}f(\rho)\,d\rho\qquad\qquad
\forall 0\leq r< R_0.
\end{aligned}
\end{equation}
Then by \eqref{decreasing-f} and  \eqref{main-eq-integral-1},
\begin{equation}\label{f'-upperbound-1}
|f'(r)|\leq\left(\frac{|m|}{1+m}+1\right)rf(0)^{2-m}
\leq\left(-\frac{m}{1+m}+1\right)R_0f(0)^{2-m}\quad\forall r\in[0,R_0).
\end{equation}
Multiplying \eqref{main-eq-integral-1} by $f^{-m}$ and integrating,
\begin{equation*}
\ln\left(\frac{f(r)}{f(0)}\right)= \frac{m}{1+m}\int_0^{r}\rho 
f(\rho)^{1-m}\,d\rho-\int_0^r\left[f(s)^{-m}\int_0^sf(\rho)\,d\rho\right]
ds\quad\forall 0\le r<R_0.
\end{equation*}
Hence
\begin{equation*}
\begin{aligned}
\left|\ln\left(\frac{f(r)}{f(0)}\right)\right|&\leq \frac{1}{2}
\left(-\frac{m}{1+m}+1\right)r^2f(0)^{1-m}\\
&\leq  \frac{1}{2}\left(-\frac{m}{1+m}+1\right)R_0^2f(0)^{1-m}:=C_1
\quad (\mbox{say})\quad\forall 0\le r<R_0.
\end{aligned}
\end{equation*}
Thus
\begin{equation}\label{f-lowerbound-1}
  f(r)\geq f(0)e^{-C_1} \qquad \forall r\in [0,R_0).
\end{equation}
By \eqref{decreasing-f}, \eqref{f'-upperbound-1} and \eqref{f-lowerbound-1}, 
a contradiction arises. Hence no such sequence $\{r_i\}_{i=1}^{\infty}$ 
exists. Therefore $R_0=\infty$ and there exists a unique solution $f$ of 
\eqref{main-symmetric-1}, \eqref{condition-boundary-1}, in $(0,\infty)$.
\end{proof}

\begin{lem}\label{lem-r-f-bound-1}
Let $\eta>0$ and $f$ be the solution of \eqref{main-symmetric-1},
\eqref{condition-boundary-1}, in $(0,\infty)$. Then
\begin{equation}\label{bound-r-f-1}
0<r^{\frac{2}{1-m}}f(r)<\left(\frac{2(1+m)}{1-m}\right)^{\frac{1}{1-m}}
\quad\forall r>0.
\end{equation}
\end{lem}
\begin{proof}
We will use an argument similar to the proof of Lemma 2.1 in \cite{Hs3} to 
prove \eqref{bound-r-f-1}. By Lemma \ref{lem-property-f}, $f'<0$ in 
$(0,\infty)$. Hence by \eqref{main-eq-integral-1},
\begin{equation*}
-\frac{f'}{f^{2-m}}\geq -\frac{m}{1+m}r+r= \frac{1}{1+m}r\quad\forall r>0.
\end{equation*}
Integrating over $(0,r)$ and simplifying,
\begin{equation*}
f(r)\leq \left(\frac{1-m}{2(1+m)}r^2+f^{m-1}(0)\right)^{-\frac{1}{1-m}}
<\left(\frac{2(1+m)}{1-m}\right)^{\frac{1}{1-m}}r^{-\frac{2}{1-m}}
\quad\forall r>0
\end{equation*}
and \eqref{bound-r-f-1} follows.
\end{proof}

\begin{lem}\label{lem-mass-2mu(T-t)-0}
For any $\mu>0$, there exists a constant $\eta>0$ and a solution $f$ of 
\eqref{main-symmetric-1}, \eqref{condition-boundary-1}, in $(0,\infty)$ 
that satisfies
\begin{equation}\label{eq-integration-f-mu}
\int_{0}^{\infty}f(r)\,dr=\mu. 
\end{equation}
\end{lem}
\begin{proof}
By Lemma \ref{lem-existence-symmetric} there exists a solution 
$\phi(r)$ of 
\begin{equation*}
\left(\frac{\phi'}{\phi^{1-m}}\right)'+\frac{1}{1+m}\phi-\frac{m}{1+m}r
\phi'=0
\end{equation*}
in $(0,\infty)$ which satisfies $\phi(0)=1$ and $\phi'(0)=0$. Moreover by
Lemma \ref{lem-r-f-bound-1},
\begin{equation}\label{eq-mass-phi-1}
\int_0^{\infty}\phi(r)\,dr:=A_1<\infty.
\end{equation}
We now choose the constant $\eta$ such that
\begin{equation}\label{eq-mass-phi-2}
A_1\eta^{\frac{1+m}{2}}=\mu
\end{equation}  
and let
\begin{equation}\label{eq-rescale-1}
f(r)=\eta \phi(\eta^{\frac{1-m}{2}}r).
\end{equation}
Then $f$ satisfies \eqref{main-symmetric-1} and \eqref{condition-boundary-1} 
in $(0,\infty)$. By \eqref{eq-mass-phi-1}, \eqref{eq-mass-phi-2} 
and \eqref{eq-rescale-1},
\begin{equation*}
\int_0^{\infty}f(r)\,dr=A_1\eta^{\frac{1+m}{2}}=\mu.
\end{equation*}
and \eqref{eq-integration-f-mu} follows.
\end{proof}

\begin{cor}\label{lem-mass-2mu(T-t)-1}
For any $\mu>0$ and $T>0$ there exists an even, smooth, positive 
solution $v(x,t)$ of \eqref{main-very-fast-diffusion-1} in $\R\times(0,T)$ which satisfies
\begin{equation*}
\int_{\R}v(x,t)\,dx=2\mu(T-t) \qquad \forall t\in [0,T).
\end{equation*}
\end{cor}
\begin{proof}
Let $\eta$ and $f$ be as in Lemma \ref{lem-mass-2mu(T-t)-0} and let
\begin{equation*}
v(x,t)=(T-t)^{\frac{1}{1+m}}f\left(|x|(T-t)^{\frac{-m}{1+m}}\right).
\end{equation*}
The $v$ is an even smooth solution of \eqref{main-very-fast-diffusion-1} in $\R\times(0,T)$ with
\begin{equation*}
\int_{\R}v(x,t)\,dx=2(T-t)\int_0^{\infty}f(r)\,dr=2\mu(T-t)
\end{equation*}
and the corollary follows.
\end{proof}

\begin{lem}\label{lem-limit-w-general}
Let $\mu>0$ and let $f$ be as in Lemma \ref{lem-mass-2mu(T-t)-0}. Then
$r^{-\frac{1}{m}}f(r)$ increases to $(\mu|m|)^{\frac{1}{m}}$ as 
$r\to\infty$. Moreover there exist constants $a>0$ and $r_0>a/(\mu |m|)$ 
such that
\begin{equation}\label{eq-f-asymptotic-infty}
(\mu |m|r+a)^{\frac{1}{m}}\le f(r)\le (\mu |m|r)^{\frac{1}{m}}
\quad\forall r\ge r_0.
\end{equation}
\end{lem}
\begin{proof}
Let $w(r)=r^{-\frac{1}{m}}f(r)$. By direct computation $w(r)$ satisfies
\begin{equation}\label{eq-w-asymptotic-1}
\left(\frac{w'}{w^{1-m}}\right)'+\frac{2}{r}\left(\frac{w'}{w^{1-m}}\right)
-\frac{m}{1+m}r^{\frac{1}{m}}w'=0 \qquad \mbox{in $(0,\infty)$}.
\end{equation}
By Lemma \ref{lem-property-f} \eqref{property-f-1} holds in $(0,\infty)$. 
Then
\begin{equation}\label{eq-w'-positive-1}
w'(r)=-\frac{1}{m}r^{-\left(\frac{1}{m}+1\right)}(f(r)-mrf'(r))>0 \qquad 
\qquad \forall r>0.
\end{equation}
Let
\begin{equation*}
g(r)=\exp\left(-\frac{m}{1+m}\int_0^r\rho^{\frac{1}{m}}w^{1-m}(\rho)
\,d\rho\right) \qquad \qquad \forall r>0.
\end{equation*}
Then $g'(r)=-\frac{m}{1+m}r^{\frac{1}{m}}w^{1-m}(r)g(r)$. Multiplying 
\eqref{eq-w-asymptotic-1} by $r^2g(r)$ and integrating over $(0,r)$,
\begin{equation}\label{eq-constant-1}
r^2g(r)\frac{w'(r)}{w^{1-m}(r)}=-\frac{\eta^m}{m}.
\end{equation}
By \eqref{eq-w'-positive-1}, 
\begin{equation}\label{eq-g-lowerbound}
g(r)\geq\exp\left(-\frac{m}{1+m}w^{1-m}(1)\int_1^2\rho^{\frac{1}{m}}
\,d\rho\right):=c \qquad \forall r\geq 2
\end{equation}
for some constant $c>0$. Hence by \eqref{eq-w'-positive-1},
\eqref{eq-constant-1}, and \eqref{eq-g-lowerbound},
\begin{equation}\label{eq-r-times-w-1}
0<\frac{w'(r)}{w^{1-m}(r)}\leq \frac{C}{r^2} \qquad \forall r\geq 2
\end{equation}
for some constant $C>0$. Thus
\begin{equation}\label{eq-limit-lower-order-1}
r\frac{w'(r)}{w^{1-m}(r)}\to 0 \qquad \mbox{as $r\to\infty$}.
\end{equation}
By \eqref{main-eq-integral-1}, \eqref{bound-r-f-1} and 
\eqref{eq-integration-f-mu},
\begin{equation}\label{eq-H'/H^1-m-1}
\lim_{r\to\infty}\frac{f'(r)}{f^{1-m}(r)}\to -\mu.
\end{equation}
Hence by \eqref{eq-limit-lower-order-1} and \eqref{eq-H'/H^1-m-1},
\begin{equation*}
0=\lim_{r\to\infty}\left[r\frac{w'(r)}{w^{1-m}(r)}\right]=\lim_{r\to\infty}
\left[-\frac{w^m}{m}+\frac{f'}{f^{1-m}}\right]
=\frac{\lim_{r\to\infty}w^m(r)}{|m|}-\mu.
\end{equation*}
Thus $\lim_{r\to\infty}w(r)=\left(\mu|m|\right)^{\frac{1}{m}}$.
By \eqref{eq-r-times-w-1} there exists a constant $a>0$ such that
\begin{equation*}
|w^m(r)-\mu|m||\le\left|\int_r^{\infty}(w^m)'(\rho)\,d\rho\right|
\le\int_r^{\infty}|m|w^{m-1}(\rho)w'(\rho)\,d\rho
\le a\int_r^{\infty}\rho^{-2}\,d\rho=a/r
\end{equation*}
for any $r>2$ and \eqref{eq-f-asymptotic-infty} follows.

\end{proof}

\begin{cor}\label{cor-limit-condition-v-1}
Let $\mu>0$, $T>0$, and let $v(x,t)$ be as in 
Corollary \ref{lem-mass-2mu(T-t)-1}. Then $|x|^{-\frac{1}{m}}v(x,t)$
increases to $(\mu |m|)^{\frac{1}{m}}$ as $|x|\to\infty$
uniformly on $0\leq t\leq T-\delta$ for any $\delta>0$. Moreover there
exist constants $a>0$ and $r_0>a/(\mu |m|)$ such that
\begin{equation*}
(\mu |m||x|+a(T-t)^{\frac{m}{1+m}})^{\frac{1}{m}}
\le v(x,t)
\le (\mu |m||x|)^{\frac{1}{m}}
\end{equation*}
holds for any $|x|\ge r_0(T-t)^{\frac{m}{1+m}}$ and $0<t<T$.
\end{cor}

\section{Properties of the Green function in $(-R,R)$}
\label{Properties of the Green function in $(-R,R)$}
\setcounter{equation}{0}
\setcounter{thm}{0}

In this section we will prove some properties of the Green function for 
the Laplace operator on $I_R=(-R,R)$. For any $R>0$ and $f\in L^1(I_R)$, 
let 
\begin{equation*}
G_R(x,y)=
\begin{cases}
-\frac{(R+y)(R-x)}{2R}\qquad\mbox{if $-R \leq y \leq x \leq R$}\\
-\frac{(R-y)(R+x)}{2R}\qquad\mbox{if $-R \leq x \leq y \leq R$}
\end{cases}
\end{equation*}
and 
\begin{equation*}
G_R(f)(x)=\int_{-R}^RG_R(x,y)f(y)\,dy.
\end{equation*}
Then
\begin{equation}\label{eq-def-G-r-1}
G_R(f)(x)=-\frac{1}{2R}\Bigg[\int_{-R}^x(R+y)(R-x)f(y)\,dy+\int_x^R
(R-y)(R+x)f(y)\,dy\Bigg].
\end{equation}

\begin{lem}
The function $G_R(x,y)$ is the Green function for the Laplacian in $[-R,R]$.
\end{lem}

\begin{proof}
By \eqref{eq-def-G-r-1} $G_R(f)(\pm R)=0$ and
\begin{equation*}
\begin{aligned}
G_R(f)(x)&=-\frac{1}{2R}\Bigg[R^2\int_{-R}^Rf(y)\,dy+R\Big(\int_{-R}^xyf(y)
\,dy-\int_x^Ryf(y)\,dy\Big)\\
&\qquad \qquad +xR\Big(\int_x^Rf(y)\,dy-\int_{-R}^xf(y)\,dy\Big)
-x\int_{-R}^Ryf(y)\,dy\Bigg].
\end{aligned}
\end{equation*}
Then by direction computation,
\begin{equation*}
(G_R(f))'(x)=-\frac{1}{2R}\left[R\left(\int_x^Rf(y)\,dy-\int_{-R}^xf(y)\,dy
\right)-\int_{-R}^Ryf(y)\,dy\right]
\end{equation*}
and
\begin{equation*}
(G_R(f))''(x)=f(x)\quad\forall f\in C(I_R),x\in I_R.
\end{equation*}
Hence the second derivatives of $G_R(x,y)$ is the Dirac delta function in a 
distribution sense. Thus the function $G_R(x,y)$ is the Greens function 
for the Laplacian in $[-R,R]$.
\end{proof}

We next introduce the operator
\begin{equation}\label{eq-define-green-1}
G_R^{\ast}(f)(x)=\int_{-R}^{R}\left[G_R(x,y)-G(0,y)\right]f(y)\,dy
\end{equation} 
where $R>0$ and $f\in L^1((-R,R))$. Note that by direct computation
\begin{equation*}
G_R^{\ast}(f'')(x)=f(x)-f(0)
\end{equation*}
for any function $f\in C^2[-R,R]$ such that $f(R)=f(-R)$.

\begin{lem}\label{G_R-aymptotic}
Let $0\leq f \in L^1(\R)$ satisfy 
\begin{equation}\label{eq-condition-f-1}
|f(x)|\leq C|x|^{\frac{1}{m}}\quad\forall |x|\geq R_0
\end{equation}
for some constant $R_0>1$. Then 
\begin{equation}\label{eq-G_R^ast-limit-1}
\left|G_R^{\ast}(f)(x)-\frac{|x|}{2}\int_{\R}f\,dx-\theta_R(x)\right|
\le\theta(x)\quad\forall |x|\leq R, R\geq R_0 
\end{equation}
for some functions $\theta_R(\cdot)\in L^1(-R,R)$ and $\theta(\cdot)
\in L_{loc}^1(\R)$ which satisfy
\begin{equation}\label{eq-G_R^ast-limit-condition}
\theta_R(x)=|x|\cdot o(R)\quad\mbox{as}\,\,\, R\to \infty
\qquad\mbox{and}\qquad 
\theta(x)=o(|x|)\quad\mbox{as} \,\,\, |x|\to \infty.
\end{equation}
\end{lem}
\begin{proof}
By direct computation,
\begin{equation*}
\begin{aligned}
G_R^{\ast}(f)(x)&=-\int_0^xyf(y)\,dy-\frac{x}{2}\left(\int_x^Rf(y)\,dy
-\int_{-R}^xf(y)\,dy\right)+\frac{x}{2R}\int_{-R}^Ryf(y)\,dy\\
&=I_1+I_2+I_3
\end{aligned}
\end{equation*}
where
\begin{equation*}
I_1=-\int_0^xyf(y)\,dy,\qquad\qquad  I_2=-\frac{x}{2}\left(\int_x^Rf(y)
\,dy-\int_{-R}^xf(y)\,dy\right)
\end{equation*}
and 
\begin{equation*}
I_3=\frac{x}{2R}\int_{-R}^Ryf(y)\,dy.
\end{equation*}
By \eqref{eq-condition-f-1}
\begin{equation*}
|yf(y)|\le C|y|^{1+\frac{1}{m}}\to 0\quad\mbox{ as }|y|\to\infty.
\end{equation*}
Hence if $\int_0^{\infty}yf(y)\,dy=\infty$, then by the l'Hospital rule,
\begin{equation}\label{eq-I1-a}
\lim_{x\to\infty}\frac{I_1}{x}=-\lim_{x\to\infty}xf(x)=0.
\end{equation}
Similarly if $\int_{-\infty}^0yf(y)\,dy=\infty$, then
\begin{equation}\label{eq-I1-b}
\lim_{x\to-\infty}\frac{I_1}{x}=0.
\end{equation}
If $yf(y)\in L^1(\R)$, then
\begin{equation}\label{eq-I1-c}
\lim_{|x|\to\infty}\frac{I_1}{|x|}=0.
\end{equation}
Similarly
\begin{equation}\label{eq-I3}
\lim_{R\to\infty}I_3=0.
\end{equation}
Now
\begin{equation*}
\left|I_2-\frac{|x|}{2}\int_{\R}f\,dy\right|
=\left\{\begin{aligned}
&\frac{|x|}{2}
\left(\int_x^Rf\,dy+\int_x^{\infty}f\,dy+\int_{-\infty}^{-R}f\,dy
\right)\quad\mbox{ if }0\le x\le R\\
&\frac{|x|}{2}
\left(\int_{-R}^xf\,dy+\int_{-\infty}^xf\,dy+\int_R^{\infty}f\,dy
\right)\quad\mbox{ if }-R\le x\le 0.
\end{aligned}\right.
\end{equation*}
Then by \eqref{eq-condition-f-1}, 
\begin{equation}\label{eq-I2}
\left|I_2-\frac{|x|}{2}\int_{\R}f\,dy\right|\leq C|x|^{2+\frac{1}{m}}
\quad\forall |x|\le R
\end{equation}
for some constant $C>0$. Let $\theta_R(x)=I_3$, $\theta (x)=I_1+e(x)$, 
where $e(x)=C|x|^{2+\frac{1}{m}}$. Since
\begin{equation*}
\lim_{|x|\to\infty}\frac{e(x)}{|x|}=0 \qquad \qquad \mbox{as $|x|\to\infty$},
\end{equation*}
by \eqref{eq-I1-a}, \eqref{eq-I1-b}, \eqref{eq-I1-c}, \eqref{eq-I3},
and \eqref{eq-I2} we get \eqref{eq-G_R^ast-limit-1} and the lemma follows.
\end{proof}

\section{Convergence of the Dirichlet solutions}
\label{Convergence of the Dirichlet solutions}
\setcounter{equation}{0}
\setcounter{thm}{0}

In this section we will use a modification of the technique of
P.~Daskalopoulos and M.A.Del Pino \cite{DP} to prove the convergence of
solutions $u^R$ of the Dirichlet problem \eqref{problem-Dirichlet}
to the solution of \eqref{eq-Cauchy-problem} that satisfies 
\eqref{eq-thm-condition-1} as $R\to\infty$.

For any $R\ge 1$, $\mu>0$, and $\3\in(0,1)$, let $u_{\3}^{R,\mu}$ be 
the unique solution of \eqref{problem-Dirichlet} with 
initial data $u_{\3}^{R,\mu}(x,0)=u_0(x)+\3$ (cf. \cite{ERV}, 
\cite{Hu1}). By an argument similar to the proof of Lemma 2.2 of 
\cite{Hu2} $u_{\3}^{R,\mu}$ satisfies the Aronson-Benilan inequality
\begin{equation}\label{Aronson-Benilan inequality}
u_t\le\frac{u}{(1-m)t}
\end{equation}
in $I_R\times (0,\infty)$.
Since by the maximum principle $0<u_{\3_1}^{R,\mu}\le u_{\3_2}^{R,\mu}$
for any $\3_2>\3_1>0$, 
$$
u^{R,\mu}=\lim_{\3 \to 0}u_{\3}^{R,\mu}
$$ 
exists. When there is no ambiguity, we will drop the superscript $\mu$ 
and write $u_{\3}^R$, $u^R$, for $u_{\3}^{R,\mu}$ and $u^{R,\mu}$
respectively.

\begin{thm}\label{thm-convergence-Dirichlet-problem}
Let $\mu>0$ and $0\le u_0\in L^{\infty}(\R)$ be such that
\begin{equation}\label{eq-condition-u_0}
u_0(x)\le (\mu_0 |m||x|)^{\frac{1}{m}}\quad\forall |x|\ge R_0
\end{equation}
for some constant $R_0>1$ and $0<\mu_0\leq \mu$. Then $u^R=u_{\3}^{R,\mu}$ converges uniformly
on every compact subset of $\R\times (0,T)$ as $R\to\infty$
to a solution $u$ of \eqref{eq-Cauchy-problem} which satisfies 
\eqref{eq-thm-condition-1} and \eqref{eq-condition-neumann-2} 
uniformly on $[a,b]$ for any $0<a<b<T$ where $T$ is given by 
\eqref{eq-mu-time}.  
\end{thm}

We will prove Theorem \ref{thm-convergence-Dirichlet-problem} in section 4.
In this section we will prove the following sequential version of 
Theorem \ref{thm-convergence-Dirichlet-problem}.

\begin{thm}\label{thm-sequential-convergence-problem}
Let $\mu>0$ and $0\le u_0\in L^{\infty}(\R)$ be such that
\eqref{eq-condition-u_0} holds for some constant $R_0>1$. 
Let $\{R_k\}$ be a sequence such that $R_k\ge 1$ for all $k\in\Z^+$ 
and $R_k\to\infty$ as $k\to\infty$. Then there exists a subsequence
$\{R_k'\}$ of $\{R_k\}$ such that $u^{R_k'}=u^{R_k',\mu}$ converges uniformly
on every compact subset of $\R\times (0,T)$ as $k\to\infty$
to a solution $u$ of \eqref{eq-Cauchy-problem} which satisfies 
\eqref{eq-thm-condition-1} 
where $T$ is given by \eqref{eq-mu-time}.  
\end{thm}

\begin{proof}
Our construction goes as follows. For any $\mu>0$, we solve the boundary 
value problem \eqref{problem-Dirichlet}
on a sequence of expanding cylindrical domains $I_{R_k}\times[0,\infty)$, 
$I_{R_k}=(-R_k,R_k)$. 
We then use the self-similar solutions constructed in section one as 
barriers in an average sense to show that the limit 
of those solutions along a subsequence of $\{R_k\}$ converges to a solution 
of \eqref{eq-Cauchy-problem} that satisfies \eqref{eq-thm-condition-1}  
as $R_k\to\infty$.

For any $0<\delta<T$ let $v^{T-\delta}$, $v^{T+\delta}$, be the 
self-similar solutions given by Corollary \eqref{lem-mass-2mu(T-t)-1} 
which satisfy
\begin{equation}\label{sp-soln-mass+}
\int_{\R}v^{T+\delta}(x,t)\,dx=2\mu(T+\delta-t) \qquad\forall 0<t<T+\delta
\end{equation}
and 
\begin{equation}\label{sp-soln-mass-}
\int_{\R}v^{T-\delta}(x,t)\,dx=2\mu(T-\delta-t) \qquad\forall 0<t<T-\delta.
\end{equation}
Since by \eqref{eq-mu-time},
\begin{equation*}
\int_{\R} u_0\,dx=2\mu T,
\end{equation*}
it follows from \eqref{sp-soln-mass+}, \eqref{sp-soln-mass-}, and
Lemma \ref{G_R-aymptotic} that there exists $R_0'\ge R_0$ and $l_{\delta}>0$
such that 
\begin{equation}\label{G_R^{ast}-time=0}
-l_{\delta}+G_R^{\ast}(v^{T-\delta}(\cdot,0))(x)\leq G_R^{\ast}(u_0)(x)
\leq G_R^{\ast}(v^{T+\delta}(\cdot,0))(x)+l_{\delta}
\quad\forall |x|\le R
\end{equation}
for any $R\ge R_0'$. Without loss of generality we may assume that 
$R_0'=R_0$ and
$R_k\ge R_0$ for all $k\in\Z^+$. We will also assume that $R\ge R_0$ 
for the rest of the paper.

We will next show that there exists a subsequence of $\{R_k\}$ 
which we will still denote by $\{R_k\}$ and a nonnegative constant 
$L_{\delta}$ such that 
\begin{equation}\label{eq-lower-upper-general-1}
-L_{\delta}+G_{R_k}^{\ast}(v^{T-\delta}(\cdot,t))(x)\leq 
G_{R_k}^{\ast}(u^{R_k}(\cdot,))(x)\leq G_{R_k}^{\ast}(v^{T+\delta}
(\cdot,t))(x)+L_{\delta}
\end{equation}
holds for any $|x|\leq R_k$, $0\leq t\leq T-3\delta$, and $k\in\Z^+$. 
We first prove the left hand side of \eqref{eq-lower-upper-general-1}. Let
\begin{equation*}
W(x,t)=G_R^{\ast}(u_{\3}^{R}(\cdot,t)-v^{T-\delta}(\cdot,t))(x).
\end{equation*} 
We will prove that $W(x,t)\geq -L_{\delta}$ for $|x|\leq R$ and 
$0\leq t\leq T-2\delta$ using the maximum principle. By direct 
computation,
\begin{equation*}
W_t=G_{R}^{\ast}\left(\left[
\frac{(u_{\3}^R)^{m}-(v^{T-\delta})^m}{m}\right]_{xx}\right)=a(x,t) 
W_{xx}-b(t)
\end{equation*}
where
\begin{equation*}
a(x,t)=\frac{(u_{\3}^{R})^m-(v^{T-\delta})^m}
{m(u_{\3}^R-v^{T-\delta})}(x,t)
\end{equation*}
and 
\begin{equation*}
b(t)=\frac{(u_{\3}^R)^{m}(0,t)-(v^{T-\delta})^m(0,t)}{m}.
\end{equation*}
Note that 
\begin{equation}\label{eq-upper-b(t)-1}
b(t)\leq \frac{\left(v^{T-\delta}\right)^m(0,t)}{|m|}\leq 
\frac{\inf_{0\leq t\leq T-2\delta}\left(v^{T-\delta}\right)^m(0,t)}{|m|}
\quad\forall 0<t\le T-2\delta.
\end{equation}
Hence
\begin{equation*}
b(t)\leq B \qquad \forall 0< t\leq T-2\delta
\end{equation*}
for some constant $B<\infty$. Therefore, if we set $\4{W}=W+Bt$, 
then $\4{W}(x,t)$ satisfies the differential inequality
\begin{equation}\label{eq-inequality-tilde-W-1}
\4{W}_t\geq a(x,t) \4{W}_{xx}\quad\mbox{ in }
I_R\times(0,T-2\delta).
\end{equation}
By \eqref{G_R^{ast}-time=0},
\begin{equation}\label{eqn-W-time=0}
\4{W}(x,0)=W(x,0)\geq -l_{\delta}\quad\forall |x|\le R.
\end{equation}
By Corollary \ref{cor-limit-condition-v-1} $|x|^{-\frac{1}{m}}
v^{T-\delta}(x,t)$ increases to $(\mu |m|)^{\frac{1}{m}}$ uniformly 
on $0\le t\le T-2\delta$ as $|x|\to \infty$. Thus
\begin{equation*}
\frac{\left(v^{T-\delta}\right)^{m}(x,t)}{m} \leq -\mu|x| 
\qquad \qquad \forall |x|> 0.
\end{equation*}
Hence
\begin{equation}\label{eqn-boundary-comparison}
\frac{\left(v^{T-\delta}\right)^{m}(x,t)}{m} \leq 
\frac{\left(u_{\3}^R\right)^{m}(x,t)}{m} \qquad \forall |x|=R>0.
\end{equation}
Since
\begin{equation*}
\4{W}_t(x,t)=\frac{\left(u^R_{\3}\right)^m 
-\left(v^{T-\delta}\right)^m}{m}(x,t)-b(t)+B\quad\forall |x|=R,
0\le t\le T-2\delta,
\end{equation*}
by \eqref{eqn-boundary-comparison},
\begin{equation*}
\4{W}_t(x,t)\geq -b(t)+B\geq 0 \quad \forall |x|=R,0\leq t\leq T-2\delta.
\end{equation*}
\begin{equation}\label{eq-inequality-tilde-W-2}
\Rightarrow \4{W}(x,t)\geq W(x,0)\geq -l_{\delta}\quad\forall |x|=R,0\leq t\leq T-2
\delta.
\end{equation}
Then by \eqref{eq-inequality-tilde-W-1}, \eqref{eqn-W-time=0},
\eqref{eq-inequality-tilde-W-2}, and the maximum principle, 
\begin{equation}\label{eq-inequality-tilde-W-left}
\4{W}(x,t)\geq -l_{\delta}, \qquad \forall |x|\leq R,\,\,0\leq t 
\leq T-2\delta.
\end{equation} 
Letting $\3 \to 0$ in \eqref{eq-inequality-tilde-W-left},
\begin{equation}\label{eq-left-hand-G^ast_R-1}
G_R^{\ast}(u^{R}(\cdot,t))(x)\geq G_R^{\ast}(v^{T-\delta}
(\cdot,t))(x)-L_{\delta}\quad\forall |x|\le R,0\le t\le T-2\delta
\end{equation}
where $L_{\delta}$ is any number greater than or equal to $l_{\delta}+BT$.

Before we show the right hand side of 
\eqref{eq-lower-upper-general-1}, we will first construct the solution 
$u$ of \eqref{main-very-fast-diffusion-1}. For any $0<r\le R$ let 
\begin{equation*}
H(r)=\frac{1}{2}\int_{|x|=r}\left[G_R(x,y)-G_R(0,y)\right]\,d\sigma(x)
=\left[\frac{G_R(r,y)+G_R(-r,y)}{2}-G_R(0,y)\right].
\end{equation*}
Then by direct computation,
\begin{equation}\label{eq-integral-conclusion-Q-1}
H(r)=\begin{cases}
    \frac{r-|y|}{2} \qquad \mbox{if $|y|<r$}\\
    \,\,\,0 \qquad \quad \mbox{if $|y|\geq r$}
  \end{cases}
\end{equation}
holds for any $0<r\le R$. Putting $x=\pm r$ and averaging on both sides of 
\eqref{eq-left-hand-G^ast_R-1}. By \eqref{eq-integral-conclusion-Q-1}, 
\begin{equation*}
\frac{1}{2}\int_{-r}^r(r-|y|)u^R(y,t)\,dy \geq \frac{1}{2}
\int_{-r}^r(r-|y|)v^{T-\delta}(y,t)\,dy-L_{\delta}\quad\forall 0<r\le R.
\end{equation*}
By integration by parts,
\begin{equation}\label{eq-average-compare-1}
\frac{1}{2}\int_0^r\left[\int_{-\rho}^{\rho}u^R(x,t)\,dx\right]
d\rho\geq \frac{1}{2}\int_0^r\left[\int_{-\rho}^{\rho}v^{T-\delta}
(x,t)\,dx\right]d\rho-L_{\delta}
\end{equation}
holds for any $0<r<R$ and $0<t\leq T-2\delta$. We now recall that the 
special solutions $v^{T\pm\delta}$ has the form
\begin{equation*}
v^{T\pm\delta}(x,t)=(T\pm\delta-t)^{\frac{1}{1+m}}f
\left(|x|(T\pm\delta-t)^{\frac{-m}{1+m}}\right)
\end{equation*}
with $\int_{0}^{\infty}f(r)\,dr=\mu$ where $f$ is given by 
Lemma \ref{lem-mass-2mu(T-t)-0}. By direct computation,
\begin{equation}\label{eq-special-soln-mass}
\frac{1}{2}\int_0^r\left[\int_{-\rho}^{\rho}v^{T\pm\delta}(x,t)
\,dx\right]d\rho=(T\pm\delta-t)^{1+\frac{m}{1+m}}\int_0^{\alpha(t)r}
\left[\int_{0}^{\rho}f(r)\,dr\right]d\rho
\end{equation}
where $\alpha(t)=(T-\delta-t)^\frac{-m}{1+m}$ for $v^{T-\delta}$
and $\alpha(t)=(T+\delta-t)^\frac{-m}{1+m}$ for $v^{T+\delta}$. 
Let $\mu>\delta'>0$ be 
a constant to be determined later. Since
\begin{equation*}
\int_{0}^{\infty}f(r)\,dr=\mu,
\end{equation*}
then there exists $R_0''\ge R_0$ such that
\begin{equation*}
\int_{0}^{\rho} f(r)\,dr \geq \mu-\delta' \qquad \rho\geq R_0''.
\end{equation*}
We now choose
\begin{equation*}
0<\delta'<\min\left(\mu,\frac{\mu\delta}{T-\delta}\right).
\end{equation*}
Then $(T-\delta-t)(\mu-\delta')\geq (T-2\delta-t)\mu$ holds for any
$0<t\le T-2\delta$. Hence
\begin{align}\label{eq-average-compare-2}
\frac{1}{2}\int_0^r\left[\int_{-\rho}^{\rho}v^{T-\delta}(x,t)
\,dx\right]d\rho &\geq (T-\delta-t)^{1+\frac{m}{1+m}}\int_{R_0''}^{\alpha(t)r}
\left[\int_{0}^{\rho}f(r)\,dr\right]d\rho\nonumber\\
&\geq  (T-\delta-t)^{1+\frac{m}{1+m}}\left(\mu-\delta'\right)
\left(\alpha(t)r-R_0''\right)\nonumber\\
&\geq \mu(T-2\delta-t)\left(r-(T-\delta-t)^{\frac{m}{1+m}}R_0''\right)  
\end{align}
holds for any $0\le t\le T-2\delta$ and $r\geq \delta^{\frac{m}{1+m}}R_0''$.
Then by \eqref{eq-average-compare-1} and \eqref{eq-average-compare-2},
\begin{equation}\label{eq-lowerbound-integral-u^R-1}
\frac{1}{2}\int_0^r\left[\int_{-\rho}^{\rho}u^R(x,t)\,dx\right]
d\rho \geq \mu(T-t-2\delta)\left(r-(T-t-\delta)^{\frac{m}{1+m}}
R_0''\right)-L_{\delta}
\end{equation}
holds for any $0\le t\le T-2\delta$ and $r\geq\delta^{\frac{m}{1+m}}R_0''$.

\vspace{6pt}

\noindent\textbf{Claim:} Given any $0<\delta<T/3$ the sequence $\{R_k\}$ 
has a subsequence still denoted by $\{R_k\}$ such that as $k\to\infty$ and $u^{R_k}$ will converge uniformly on every compact subset of 
$\R\times(0,T-3\delta]$ to a solution $u^{\delta}$ of 
\eqref{main-very-fast-diffusion-1} in $\R\times(0,T-3\delta)$ that satisfies
\eqref{eq-thm-condition-1} for any $0\le t\le T-3\delta$.

\indent To prove the claim, we first observe that there exists 
$x_0\in\R$ such that
\begin{equation}\label{eq-limsup-nonzero-1}
\limsup_{R_k\to \infty}u^{R_k}(x_0, T-(5/2)\delta)>0.
\end{equation}
Indeed, if $\limsup_{k\to\infty}u^{R_k}(x, T-(5/2)\delta)=0$ for 
all $x\in\R$, then by the Lebesque Dominated Convergence Theorem,
\begin{equation*}
\limsup_{k\to\infty}\frac{1}{2}\int_0^r\left[\int_{-\rho}^{\rho}u^{R_k}
(x,T-(5/2)\delta)\,dx\right]d\rho=0\quad\forall r>0
\end{equation*}
which contradicts \eqref{eq-lowerbound-integral-u^R-1} since the right
hand side of \eqref{eq-lowerbound-integral-u^R-1} is strictly positive
for $0\le t\le T-(5/2)\delta$ and $r$ sufficiently large. Hence \eqref{eq-limsup-nonzero-1} holds for some $x_0\in\R$. It then follows from \eqref{eq-limsup-nonzero-1} that there exists $x_0\in\R$, 
a subsequence of $\{u^{R_k}\}$ which we still denoted by $\{u^{R_k}\}$,
and a constant $c>0$ such that
\begin{equation}\label{eq-point-lowerbound-1}
u^{R_k}(x_0,T-(5/2)\delta)\geq c_0>0 \qquad \forall k\in\Z^+
\end{equation}
for some constant $c_0>0$.
For any $r_0>0$ and $s_0\in(0,T-3\delta)$, let $K(r_0,s_0)=\overline{I_{r_0}(x_0)}
\times[s_0,T-3\delta]$. Since $u_{\3}^R$ satisfies the Aronson-Benilan 
inequality \eqref{Aronson-Benilan inequality}, by Lemma 3.2 of \cite{Hu4} 
and an argument similar to the proof of Lemma 2.8 of \cite{Hu1}
we have the following Harnack type estimate.
For any $r_0>0$, $\delta_1>0$, and $s_0\in(0,T-3\delta)$, there exist 
constants $C_1>0$ and $C_2>0$ depending on $m$, $T$, $\delta$, $\delta_1$
and $\|u_0\|_{L^{\infty}}$ such that
\begin{equation}\label{eq-harnack-type-1}
u_{\3}^R(y,t)\geq (C_1(u_{\3}^R)^m(x_0,T-(5/2)\delta)+C_2)^{\frac{1}{m}}
\end{equation} 
holds for any $(y,t)\in K(r_0,s_0)$ and $R\ge r_0+\delta_1$. Letting
$\3\to 0$ in \eqref{eq-harnack-type-1},
\begin{equation}\label{eq-harnack-type-2}
u^R(y,t)\geq (C_1(u^R)^m(x_0,T-(5/2)\delta)+C_2)^{\frac{1}{m}}
\end{equation} 
holds for any $(y,t)\in K(r_0,s_0)$ and $R\ge r_0+\delta_1$.
By \eqref{eq-point-lowerbound-1} and \eqref{eq-harnack-type-2},
\begin{equation}\label{eq-K-lowerbound-1}
u^{R_k}(y,\tau)\geq c(K(r_0,s_0))>0 \qquad \forall R_k\ge r_0+\delta_1,
(y,t)\in K(r_0,s_0)
\end{equation}
for some constant $c(K(r_0,s_0))$. Hence the sequence $\{u^{R_k}\}$
is uniformly bounded below by some positive constant on any compact 
subset of $\R\times (0,T-3\delta]$ for all $k$ sufficiently large. 
Since the sequence $\{u^{R_k}\}$ is uniformly bounded from above by 
$\|u_0\|_{\infty}$, by the Schauder estimates for parabolic equations 
\cite{LSU}  the sequence $\{u^{R_k}\}$ is equi-H\"older continuous on 
every compact subsets of $\R\times(0,T-3\delta]$. Hence by the Ascoli 
Theorem and a diagonalization argument there exists a subsequence we will still denoted by $\{u^{R_k}\}$ 
that converges uniformly on every compact subsets of $\R\times(0,T-3\delta]$ 
to a solution $u^{\delta}$ of \eqref{main-very-fast-diffusion-1} 
in $\R\times(0,T-3\delta]$. 

It remains to show that 
\begin{equation*}
u^{\delta}(\cdot,t)\to u_0\qquad\mbox{in $L^1(\R)$}\quad 
\mbox{as $t \to 0$}. 
\end{equation*}
Since $u_{\3}^R$ satisfies \eqref{Aronson-Benilan inequality}, 
$u^R$ satisfies \eqref{Aronson-Benilan inequality}. 
By \eqref{Aronson-Benilan inequality} for $u^{R_k}$ and
\eqref{eq-point-lowerbound-1},
\begin{equation}\label{eq-aronson-benilan-estimate-1}
u^{R_k}(x,t)\ge\frac{t^{\frac{1}{1-m}}}{(T-(5/2)\delta)^{\frac{1}{1-m}}}
u^{R_k}(x,T-(5/2)\delta)
\ge c_0\frac{t^{\frac{1}{1-m}}}{(T-(5/2)\delta)^{\frac{1}{1-m}}}
\end{equation}
holds for any $|x|\le R_k$, $0\leq t\le T-3\delta$ and $k\in\Z^+$.
Letting $k\to\infty$ in \eqref{eq-aronson-benilan-estimate-1},
\begin{equation}\label{eq-aronson-benilan-estimate}
u^{\delta}(x,t)
\ge c_0\frac{t^{\frac{1}{1-m}}}{(T-(5/2)\delta)^{\frac{1}{1-m}}}
\quad\forall x\in\R,0\le t\le T-3\delta.
\end{equation}
Thus for any $\psi \in C^{\infty}_0(\R)$, 
by \eqref{eq-aronson-benilan-estimate},
\begin{equation*}
\begin{aligned}
\left|\int_{\R}u^{\delta}(x,t)\psi(x)\,dx-\int_{\R}u_0(x)\psi(x)
\,dx\right|&=\left|\int_0^t\int_{\R}(u^{\delta})_t(x,s)\psi(x)\,dxds
\right|\\
&=\left|\int_0^t\int_{\R}\left(\frac{(u^{\delta})^m(x,s)}{m}\right)_{xx}
\psi(x)\,dxds\right|\\
&=\left|\int_0^t\int_{\R}\frac{(u^{\delta})^m(x,s)}{m}\psi_{xx}(x)\,dxds
\right|\\
&\leq \int_0^t\int_{\R}\left|\frac{(u^{\delta})^m(x,s)}{m}\right|\left|
\psi_{xx}(x)\right|\,dxds\\
&\leq C\int_0^ts^{\frac{m}{1-m}}\,ds\\
&=C(1-m)t^{\frac{1}{1-m}}\\
&\to 0 \qquad \mbox{as $t\to 0$}
\end{aligned}
\end{equation*}  
Hence $u^{\delta}(\cdot,t)\to u_0$ weakly in $L^1(\R)$ as $t\to 0$.
Then any sequence $\{t_i\}$, $t_i\to 0$ as $i\to 0$, has a subsequence
which we still denote by $\{t_i\}$ such that $u^{\delta}(\cdot,t_i)\to u_0$
a.e. in $\R$ as $i\to\infty$. 

Let $\phi (x):=(\mu_0 |m|(|x|-R_0))^{\frac{1}{m}}$. We claim that 
\begin{equation}\label{eq-claim-2}
u^{\delta}(x,t)\le \phi (x)=(\mu_0 |m|(|x|-R_0))^{\frac{1}{m}}
\quad\forall |x|\ge R_0, \,\, 0<t\leq T-3\delta. 
\end{equation}
Suppose the claim holds. Since $\phi (x)\in L^1((-\infty,-2R_0)\cup 
(2R_0,\infty))$ and $u^{\delta}\le\|u_0\|_{L^{\infty}}$, by the 
Lebesgue dominated convergence theorem $u^{\delta}(\cdot,t_i)\to u_0$ in 
$L^1(\R)$ as $i\to\infty$. Since the sequence $\{t_i\}$ is arbitrary, 
$u^{\delta}(\cdot,t)\to u_0$ in $L^1(\R)$ as $t\to 0$. Hence $u^{\delta}$
is a solution of \eqref{eq-Cauchy-problem} in $\R\times (0,T-3\delta)$.

We will now prove the above claim. Let $R>R_0$ and 
$$
0<\delta_1<\min((R-R_0)/2,(\|u_0\|_{L^{\infty}}+1)^m/(\mu_0 |m|)).
$$ 
Then 
$$
\phi(\pm (R_0+\delta_1))\ge \|u_0\|_{L^{\infty}}+1\ge u_{\3}^R
(\pm (R_0+\delta_1))\quad\mbox{ and }\quad\phi(\pm R)\ge u_{\3}^R(\pm R)
$$
for any $0<\3<1$. Hence by \eqref{eq-condition-u_0} and an argument 
similar to the proof of Lemma 2.3 of \cite{DK} and Lemma 2.5 of 
\cite{Hu3}, for any $0<\3<1$, 
\begin{align}\label{eq-bound-infty}
\int_{R_0+\delta_1\le |x|\le R}(u_{\3}^R(x,t)-\phi(x))_+\,dx
\le&\int_{R_0+\delta_1\le |x|\le R}(u_{\3}^R(x,t_1)-\phi(x))_+\,dx
\quad\forall t>t_1>0\nonumber\\
\to&\int_{R_0+\delta_1\le |x|\le R}(\3+u_0-\phi)_+\,dx
\quad\mbox{ as }t_1\to 0\nonumber\\
\le&2\3 (R-R_0-\delta_1)\qquad\qquad\forall t>0.
\end{align} 
Letting $\3\to 0$ and $\delta_1\to 0$ in \eqref{eq-bound-infty},
\begin{equation*}
\int_{R_0\le |x|\le R}(u^R(x,t)-\phi(x))_+\,dx\le 0\quad\forall
t>0.
\end{equation*}
Hence
\begin{equation}\label{eq-u^R-bound}
u^R(x,t)\le\phi(x)=(\mu |m|(|x|-R_0))^{\frac{1}{m}}
\quad\forall R_0\le |x|\le R, t>0.
\end{equation}
Putting $R=R_k$ in \eqref{eq-u^R-bound} and letting $k\to\infty$ we get
\eqref{eq-claim-2} and the claim follows.

We will now prove the right hand side of \eqref{eq-lower-upper-general-1}. 
Let 
\begin{equation*}
Z(x,t)=G_R^{\ast}\left(u^{R_k}_{\3}(\cdot,t)-v^{T+\delta}(\cdot,t)
\right)(x).
\end{equation*}
Then $Z(x,t)$ satisfies the equation $Z_t=d(x,t)Z_{xx}-e(t)$ with
\begin{equation*}
d(x,t)=\frac{(u_{\3}^{R_k})^m-(v^{T+\delta})^m}
{m(u_{\3}^{R_k}-v^{T+\delta})}(x,t)
\end{equation*}
and 
\begin{equation*}
e(t)=\frac{(u_{\3}^{R_k})^{m}(0,t)-(v^{T+\delta})^m(0,t)}{m}.
\end{equation*}
Since $u_{\3}^{R_k}\ge u^{R_k}$, by \eqref{eq-aronson-benilan-estimate-1}, 
\begin{equation*}
e(t)\geq -\frac{(u_{\3}^{R_k})^{m}(0,t)}{|m|}\geq 
-\frac{c_0^mt^{\frac{m}{1-m}}}{|m|(T-\frac{5\delta}{2})^{\frac{m}{1-m}}}
=-Dt^{\frac{m}{1-m}}\quad\forall 0\le t\le T-3\delta
\end{equation*}
where $D=c_0^m/(|m|(T-\frac{5\delta}{2})^{\frac{m}{1-m}})$.
Therefore, if we set $\4{Z}=Z-D\int_{0}^{t}s^{\frac{m}{1-m}}\,ds$, then 
$\4{Z}(x,t)$ satisfies
\begin{equation*}
\4{Z}_t\leq d(x,t)\4{Z}_{xx}\quad\forall |x|\le R_k, 0\le t
\le T-3\delta.
\end{equation*}
At $t=0$ we have $\4{Z}(x,0)=Z(x,0)\leq l_{\delta}$. Now
\begin{equation}\label{eq-tilde-Z-time-derivative-1}
\4{Z}_t=\frac{(u_{\3}^{R_k})(x,t)-(v^{T+\delta})^m(x,t)}{m}-e(t)
-Dt^{\frac{m}{1-m}}
\le\frac{(v^{T+\delta})^m(x,t)-(u_{\3}^{R_k})(x,t)}{|m|}.
\end{equation}
By Corollary \ref{cor-limit-condition-v-1} there exist constants $a>0$
and $r_0>a/(\mu |m|)$ such that
\begin{equation*}
v^{T+\delta}(x,t)\ge (\mu |m||x|+a(T+\delta-t)^{\frac{m}{1+m}}
)^{\frac{1}{m}}
\end{equation*}
holds for any $|x|\ge r_0(T+\delta-t)^{\frac{m}{1+m}}$ and $0<t<T+\delta$.
Hence 
\begin{equation}\label{eq-v-u-boundary-compare-1}
(v^{T+\delta})^m(x,t)\le\mu |m||x|+a(T+\delta-t)^{\frac{m}{1+m}}
\end{equation}
for any $|x|\ge r_0(4\delta)^{\frac{m}{1+m}}$ and $0<t\le T-3\delta$.
By passing to a subsequence if necessary we may assume without loss of
generality that $R_k\ge r_0(4\delta)^{\frac{m}{1+m}}$ for all $k\in\Z^+$.
Then by \eqref{eq-v-u-boundary-compare-1},
\begin{equation}\label{eq-v-u-boundary-compare-2}
((v^{T+\delta})^m-u^m)(\pm R_k,t)\le a(T+\delta)^{\frac{m}{1+m}}
\quad\forall 0\le t\le T-3\delta.
\end{equation}
By \eqref{eq-tilde-Z-time-derivative-1} and \eqref{eq-v-u-boundary-compare-2},
\begin{equation*}
\4{Z}_t(\pm R_k,t)\le a(T+\delta)^{\frac{m}{1+m}}
\quad\forall 0\le t\le T-3\delta.
\end{equation*}
Let $\hat{Z}=\4{Z}-a(T+\delta)^{\frac{m}{1+m}}t$. Then
$\hat{Z}(x,t)$ satisfies
\begin{equation*}
\hat{Z}_t\leq d(x,t)\hat{Z}_{xx}\quad\forall |x|\le R_k, 0\le t
\le T-3\delta,
\end{equation*}
\begin{equation*}
\hat{Z}(x,0)=\4{Z}(x,0)\le l_{\delta}\quad\forall |x|\le R_k,
\end{equation*}
and
\begin{align*}
&\hat{Z}_t(\pm R_k,t)\le 0\quad\forall 0\le t\le T-3\delta\\
\Rightarrow\quad&\hat{Z}(\pm R_k,t)\le \hat{Z}(\pm R_k,0)\le l_{\delta}
\quad\forall 0\le t\le T-3\delta.
\end{align*}
Then by the maximum principle $\hat{Z}\leq l_{\delta}$ in $(-R_j,R_j)\times
(0,T-3\delta)$, which implies the right hand side 
\eqref{eq-lower-upper-general-1} with 
$$
L_{\delta}=l_{\delta}+\max(BT,(1-m)DT^{\frac{1}{1-m}}
+a(T+\delta)^{\frac{m}{1+m}}T).
$$
Now by putting $x=\pm r$, $r>0$, into the right hand side of  
\eqref{eq-lower-upper-general-1} and averaging we get after simplifying
as before that
\begin{equation*}
\frac{1}{2}\int_0^r\left[\int_{-\rho}^{\rho}u^{R_k}(x,t)\,dx\right]
d\rho\le\frac{1}{2}\int_0^r\left[\int_{-\rho}^{\rho}v^{T-\delta}
(x,t)\,dx\right]d\rho+L_{\delta}
\end{equation*}
holds for any $0<r<R_k$, $0<t\leq T-3\delta$ and $k\in\Z^+$.
Letting $k\to\infty$,
\begin{equation}\label{eq-u-mass-inequality-1}
\frac{1}{2}\int_0^r\left[\int_{-\rho}^{\rho}u^{\delta}(x,t)\,dx\right]
d\rho\le\frac{1}{2}\int_0^r\left[\int_{-\rho}^{\rho}v^{T-\delta}
(x,t)\,dx\right]d\rho+L_{\delta}
\end{equation}
holds for any $r>0$ and $0<t\leq T-3\delta$. By \eqref{eq-special-soln-mass},
\begin{equation}\label{eq-v-mass-inequality-1}
\frac{1}{2}\int_0^r\left[\int_{-\rho}^{\rho}v^{T+\delta}(x,t)
\,dx\right]d\rho\le (T+\delta-t)\mu r.
\end{equation}
By \eqref{eq-u-mass-inequality-1} and \eqref{eq-v-mass-inequality-1},
\begin{equation}\label{eq-upper-integral-u^R} 
\frac{1}{2}\int_0^r\left[\int_{-\rho}^{\rho}u^{\delta}(x,t)\,dx\right]
d\rho\le (T+\delta-t)\mu r+L_{\delta}
\end{equation}
holds for any $r>0$ and $0<t\leq T-3\delta$.
By \eqref{eq-lowerbound-integral-u^R-1} and \eqref{eq-upper-integral-u^R}    
the solution $u^{\delta}$ satisfies
\begin{equation}\label{eq-upper-lower-last-1}
-\frac{L_{\delta}}{r}+\mu(T-t-2\delta)\left(1-\frac{a_0}{r}\right) \leq 
\frac{1}{2r}\int_0^r\left[\int_{-\rho}^{\rho}u^{\delta}(x,t)\,dx\right]
d\rho \leq \mu(T-t+2\delta)+\frac{L_{\delta}}{r}
\end{equation}
for all $r\ge\delta^{\frac{m}{1+m}}R_0'$ and $0\leq t \leq T-3\delta$
where $a_0=(2\delta)^{\frac{m}{1+m}}R_0''$. Now for any bounded non-negative 
integrable function $h$ on $\R$, we have
\begin{equation*}
A_r:=\frac{1}{2r}\int_0^r\left[\int_{-\rho}^{\rho}h\,dx\right]d\rho \leq 
\frac{1}{2}\|h\|_{L^1} \qquad \forall r>0
\end{equation*}
and 
\begin{equation*}
\frac{1}{2r}\int_0^r\left[\int_{-\rho}^{\rho}h\,dx\right]d\rho \geq 
\frac{r-R_1}{2r}\int_{-R_1}^{R_1}h\,dx \qquad \forall r\geq R_1>0.
\end{equation*}
Then
\begin{equation*}
\frac{1}{2}\int_{-R_1}^{R_1}h\,dx \leq \liminf_{r\to\infty} A_r \leq 
\limsup_{r\to\infty} A_r \leq \frac{1}{2}\|h\|_{L^1} \qquad \forall R_1>0.
\end{equation*}
Letting $R_1\to\infty$,
\begin{equation*}
\lim_{r\to\infty}\frac{1}{2r}\int_0^r\left[\int_{-\rho}^{\rho}h\,dx\right]
d\rho = \frac{1}{2}\int_{\R}h(x)\,dx \qquad \forall 0\leq h \in L^1(\R).
\end{equation*}
Since $u^{\delta}\le\|u_0\|_{L^{\infty}}$, letting $r\to\infty$ in
\eqref{eq-upper-lower-last-1} by \eqref{eq-claim-2} and the Lebesgue 
dominated convergence theorem,
\begin{equation}\label{eq-lower-upper-final-1}
\mu(T-t-2\delta)\leq \frac{1}{2}\int_{\R}u^{\delta}(x,t)\,dx\leq 
\mu(T-t+2\delta).
\end{equation}
It remains to construct a solution $u$ of the problem 
\eqref{eq-Cauchy-problem} which is defined up to time $T$ and satisfies 
\eqref{eq-thm-condition-1}. Let $\{\delta_k\}_{k=1}^{\infty}$ 
be a decreasing sequence 
of positive numbers such that $\delta_k\to 0$. By the previous argument
there exists a subsequence $\{R_k^1\}$ of $\{R_k\}$ such that 
$u^{R_k^1}$ converges to a solution $u^{\delta_1}$ of 
\eqref{eq-Cauchy-problem} uniformly on every compact subsets of $\R\times 
(0,T-3\delta_1]$ as $k\to\infty$. 

We construct $u^{\delta_k}$ inductively. For any $j\ge 1$ suppose 
$\{R_k^j\}_{k=1}^{\infty}$ is a subsequence of 
$\{R_k^{j-1}\}_{k=1}^{\infty}$ such that $u^{R_k^j}$
converges to a solution $u^{\delta_j}$ of \eqref{eq-Cauchy-problem} in $\R\times(0,T-3\delta_j)$   
uniformly on every compact subsets of $\R\times (0,T-3\delta_j]$ 
as $k\to\infty$. By repeating the above argument the sequence 
$\{R_k^j\}_{k=1}^{\infty}$ has a subsequence 
$\{R_k^{j+1}\}_{k=1}^{\infty}$ such that $u^{R_k^{j+1}}$ 
converges uniformly to some solution $u^{\delta_{j+1}}$ 
of \eqref{eq-Cauchy-problem} in $\R\times(0,T-3\delta_{j+1})$ on every compact subsets of $\R\times(0,T-3\delta_{j+1}]$ 
as $k\to\infty$. 

By construction we have $u^{\delta_j}=u^{\delta_{j-1}}$ on 
$\R\times(0,T-3\delta_j]$ for any $j\ge 1$. Hence
if we define $u(x,t)=u^{\delta_j}(x,t)$ for any $x\in\R$, 
$0<t\leq T-3\delta_j$, and $j\ge 1$, then $u$ satisfies 
\eqref{eq-Cauchy-problem} on $\R\times(0,T)$. Putting $\delta=\delta_j$ in 
\eqref{eq-lower-upper-final-1} and letting $j\to\infty$ we get that 
$u$ satisfies \eqref{eq-thm-condition-1} and the theorem follows.

\end{proof}

By the construction of solution of \eqref{eq-Cauchy-problem} in 
Theorem \ref{thm-sequential-convergence-problem} we have the following 
two corollaries.

\begin{cor}
For any $\mu_2>\mu_1>0$ and $0\le u_0\in L^{\infty}(\R)$ such that
\eqref{eq-condition-u_0} holds for some constants $R_0>1$ and $0<\mu_0\leq \mu_1$
with $\mu=\mu_2$ if $u_{\mu_1}$ and 
$u_{\mu_2}$ are the solutions of \eqref{eq-Cauchy-problem} in $\R\times (0,T_{\mu_1})$
and $\R\times (0,T_{\mu_2})$ resepctively given by 
Theorem \ref{thm-sequential-convergence-problem} which satisfies \eqref{eq-thm-condition-1} with
$\mu=\mu_1,\mu_2$ in $\R\times (0,T_{\mu_1})$ and $\R\times (0,T_{\mu_2})$ 
resepctively where $T_{\mu_i}$ is given by \eqref{eq-mu-time} with $\mu=\mu_1,\mu_2$
respectively, then $u_{\mu_2}\le u_{\mu_1}$ in $\R\times (0,T_{\mu_2})$.  
\end{cor}

\begin{cor}[cf. \cite{ERV}]\label{cor-global-soln}
Let $0\le u_0\in L^{\infty}(\R)$ be an even function such that 
\eqref{eq-condition-u_0} holds for some constant $R_0>1$
and $\mu_0>0$ For any $\mu\ge\mu'>0$ let $u_{\mu'}$ be the solution of 
\eqref{eq-Cauchy-problem} in $\R\times (0,T_{\mu'})$ given by 
Theorem \ref{thm-sequential-convergence-problem} 
which satisfies \eqref{eq-thm-condition-1} where $T_{\mu'}$ is given by \eqref{eq-mu-time} with $\mu$
being replaced by $\mu'$. Then $u_{\mu'}$
will increase and converge to the global solution $u$ of \eqref{eq-Cauchy-problem} 
in $\R\times (0,\infty)$ which satisfies
\begin{equation*}
\int_{\R}u(x,t)\,dx=\int_{\R}u_0(x)\,dx\quad\forall t>0.
\end{equation*}
as $\mu'\to 0$.
\end{cor}

\section{Uniqueness of solution}
\setcounter{equation}{0}
\setcounter{thm}{0}

In this section we will use a modification of the technique of \cite{Hs2}
to prove that the solution of \eqref{eq-Cauchy-problem} constructed in 
section three by the 
Dirichlet approximation and the solution of \eqref{eq-Cauchy-problem} 
constructed in \cite{Hu3} by the Neumann approximation are equal. 
We will also prove the convergence of solutions of \eqref{problem-Dirichlet}
and \eqref{problem-general-Dirichlet} as $R\to\infty$.

We first observe that by Theorem \ref{thm-sequential-convergence-problem},
Corollary \ref{cor-global-soln},
and an argument similar to the proof of Theorem 1.3 of \cite{DP} we 
have the following two results.

\begin{lem}\label{lemma-step-initial-1}
Suppose $0\le u_0\in L^{\infty}(\R)$ satisfies \eqref{eq-condition-u_0}
for some constants $\mu_0>0$, $R_0>1$, and 
$f=\sum_{i=1}^{i_0}\mu_i \chi_{I_i}$ is a step function on $[0,T_0)$ 
where $0=a_0<a_1<\cdots<a_{i_0}=T_0$ is a partition of the interval 
$[0,T_0]$, $I_i=[a_{i-1},a_i)$, $\mu_i\geq 0$ for all 
$i=1,2,\cdots,i_0$ such that 
\begin{equation*}
2\sum_{i=1}^{i_0}\mu_i(a_i-a_{i-1})\geq \int_{\R}u_0\,dx.
\end{equation*}
Let $T\in (0,T_0]$ be given by \eqref{eq-def-T-1} and 
$a_{j_0-1}<T\le a_{j_0}$ for some $j_0\in \{1,2,\dots,i_0\}$. 
Let $u_1$ be the solution of \eqref{main-very-fast-diffusion-1} in 
$\R\times(0,a_1)$ given by Theorem \ref{thm-sequential-convergence-problem} 
or Corollary \ref{cor-global-soln}
which satisfies \eqref{eq-thm-condition-1} with $f=\mu_1$ and 
$u_1(\cdot,t)\to u_0$ in $L^1(\R)$ as $t\to 0$. For each 
$i=2,3,\cdots,j_0-1$, let $u_i$ be the solution of
\eqref{main-very-fast-diffusion-1} in $\R\times(0,a_i-a_{i-1})$ 
given by Theorem \ref{thm-sequential-convergence-problem} 
or Corollary \ref{cor-global-soln} which satisfies 
\eqref{eq-thm-condition-1} with $f=\mu_i$,
$u_0=u_{i-1}$, 
and $u_i(\cdot,t)\to u_{i-1}(x,a_{i-1})$ in $L^1(\R)$ as $t\to 0$. 
Let $u_{j_0}$ be the solution of
\eqref{main-very-fast-diffusion-1} in $\R\times(0,T-a_{j_0-1})$ 
given by Theorem \ref{thm-sequential-convergence-problem} 
or Corollary \ref{cor-global-soln} which satisfies 
\eqref{eq-thm-condition-1} with $f=\mu_i$, $u_0=u_{j_0-1}$, 
and $u_{j_0}(\cdot,t)\to u_{j_0-1}(x,a_{{j_0}-1})$ in $L^1(\R)$ as $t\to 0$. 
Then the function $u$ defined by $u(x,t)=u_i(x,t-a_{i-1})$ for 
$x\in\R$, $t\in[a_{i-1},a_i)$, $i=1,2,\cdots,i_0$, is a solution 
of \eqref{eq-Cauchy-problem} in $\R\times(0,T)$ which
satisfies \eqref{main-integral-condition-1}. 
\end{lem}

\begin{lem}\label{lemma-step-initial-2}
Suppose $0\le u_0\in L^{\infty}(\R)$ satisfies \eqref{eq-condition-u_0}
for some constants $\mu_0>0$, $R_0>1$, and $0\le f\in C([0,\infty))$. 
For any $k=1,2,\cdots,$ let 
$f_k=\sum_{i=1}^{2^k}\mu_i\chi_{I_i}$ where $\mu_i=\sup_{I_i}f$, 
$I_i=[a_{i-1},a_i)$, $a_0=0$, $a_i=iT/2^{k}$, for all $i=1,2,\cdots,2^k$, 
and $T$ is given \eqref{eq-def-T-1}. Let $v_k$ be the solution of 
\eqref{eq-Cauchy-problem} in $\R\times(0,T_k)$ given by 
Lemma \ref{lemma-step-initial-1} which satisfies 
\eqref{main-integral-condition-1} in $\R\times(0,T_k)$
with $f$ being replaced by $f_k$ where $T_k$ 
is given by \eqref{eq-def-T-1} with $f=f_k$. Then $v_{k+1}\geq v_k$ on 
$\R\times(0,T_k)$ for all $k=1,2,\cdots,$ and as $k\to\infty$
$v_k$ will converge 
uniformly on every compact subset of $\R\times(0,T)$ to a solution 
$u$ of \eqref{eq-Cauchy-problem} in $\R\times(0,T)$ that satisfies 
\eqref{main-integral-condition-1}.
\end{lem}

\begin{lem}\label{lemma-step-initial-3}
Suppose $0\le u_0\in L^{\infty}(\R)$ satisfies \eqref{eq-condition-u_0}
for some constants $\mu_0>0$, $R_0>1$. If $u$ is the solution of 
\eqref{eq-Cauchy-problem} in $\R\times (0,T_{\mu})$ given by 
Theorem \ref{thm-sequential-convergence-problem} which satisfies
\eqref{eq-thm-condition-1} where $T_{\mu}$ is given by \eqref{eq-mu-time}, 
then $u$ satisfies \eqref{eq-condition-neumann-2}
uniformly on $[a,b]$ for any $0<a<b<T_{\mu}$.
\end{lem}
\begin{proof}
Let $\{R_k\}_{k=1}^{\infty}$ be a sequence of positive numbers
such that $R_k\to\infty$ as $k\to\infty$. By the proof of 
Theorem \ref{thm-sequential-convergence-problem} 
the sequence $\{R_k\}_{k=1}^{\infty}$ has a subsequence which we still 
denote by $\{R_k\}_{k=1}^{\infty}$ such that the sequence of solution 
$\{u^{R_k,\mu}\}_{k=1}^{\infty}$ of \eqref{problem-Dirichlet} with
$R=R_k$ converges uniformly on every 
compact subset of $\R\times(0,T_{\mu})$ to $u$ as $k\to\infty$. 

For any $\mu_2>\mu>\mu_1>0$, let $\4{u}_{\mu_1}$, $\4{u}_{\mu_2}$, be the 
solutions of \eqref{eq-Cauchy-problem} in $\R\times (0,T_{\mu_1})$ and 
$\R\times (0,T_{\mu_2})$ 
respectively constructed by the Neumann approximation given by 
Theorem 4.6 of \cite{Hu3} where $T_{\mu_1}$, $T_{\mu_2}$ is given by 
\eqref{eq-mu-time} with $\mu=\mu_1,\mu_2$ respectively. Then by \cite{Hu3} 
$\4{u}_{\mu_1}$ and $\4{u}_{\mu_2}$ satisfy \eqref{eq-thm-condition-1} 
with $\mu =\mu_1, \mu_2$ respectively and \eqref{eq-condition-neumann-2} 
with $\mu =\mu_1, \mu_2$ uniformly on $[a,b]$ for any $0<a<b<T_{\mu_i}$, 
$i=1,2$, respectively. Moreover $T_{\mu_1}=\int_{\R}u_0\,dx/2\mu_1
>T_{\mu}>T_{\mu_2}=\int_{\R}u_0\,dx/2\mu_2$.

Since $\4{u}_{\mu_1}$ satisfies \eqref{eq-condition-neumann-2} 
with $\mu=\mu_1$ and $\mu-\mu_1>0$, for any $0<t_1<t_2<T_{\mu}$ there 
exists $r_0>1$ such that 
\begin{equation*}
\frac{\4{u}_{\mu_1}^m}{mx}>-\mu_1-\left(\frac{\mu-\mu_1}{2}\right)
=-\frac{\mu+\mu_1}{2} \qquad \forall x\geq r_0,\,\,t_1\leq t\leq t_2
\end{equation*}
and 
\begin{equation*}
\frac{\4{u}_{\mu_1}^m}{mx}<\mu_1+\left(\frac{\mu-\mu_1}{2}\right)
=\frac{\mu+\mu_1}{2} \qquad \forall x\leq -r_0,\,\,t_1\leq t\leq t_2.
\end{equation*}
Hence
\begin{equation*}
\4{u}_{\mu_1}(x,t)\geq \left(\frac{(\mu+\mu_1)}{2}|m||x|
\right)^{\frac{1}{m}} 
\quad\forall |x|\geq r_0, \,\,t_1\leq t\leq t_2.
\end{equation*}
Thus
\begin{equation*}
\4{u}_{\mu_1}(\pm R_k,t) \geq \left(\mu|m|R_k\right)^{\frac{1}{m}}
=u_{\3}^{R_k,\mu}(\pm R_k,t)\quad\forall |x|=R_k>r_0,t_1\le t\le t_2,
0<\3<1.
\end{equation*}

Hence by \eqref{eq-condition-u_0} and an argument 
similar to the proof of Lemma 2.3 of \cite{DK} and Lemma 2.5 of 
\cite{Hu3}, for any $0<\3<1$, $t_1\le t'\le t\le t_2$,
\begin{equation}\label{eq-unique-compare-1}
\int_{|x|\le R_k}(u_{\3}^{R_k,\mu}(x,t)-\4{u}_{\mu_1}(x,t))_+\,dx
\le\int_{|x|\le R_k}(u_{\3}^{R_k,\mu}(x,t')-\4{u}_{\mu_1}(x,t'))_+\,dx
\end{equation} 
Letting $\3\to 0$ in \eqref{eq-unique-compare-1},
\begin{equation}\label{eq-unique-compare-2}
\int_{|x|\le R_k}(u^{R_k,\mu}(x,t)-\4{u}_{\mu_1}(x,t))_+\,dx
\le\int_{|x|\le R_k}(u^{R_k,\mu}(x,t')-\4{u}_{\mu_1}(x,t'))_+\,dx
\end{equation}
Since $u^{R_k,\mu}$ satisfies \eqref{eq-u^R-bound} in 
$(I_{R_k}\setminus I_{R_0})\times (0,T_{\mu})$, letting $k\to\infty$
in \eqref{eq-unique-compare-2} by the Lebesgue dominated convergence 
theorem we get
\begin{equation}\label{eq-unique-compare-3}
\int_{\R}(u(x,t)-\4{u}_{\mu_1}(x,t))_+\,dx
\le\int_{\R}(u(x,t')-\4{u}_{\mu_1}(x,t'))_+\,dx
\end{equation}
Letting $t'\to 0$ in \eqref{eq-unique-compare-3}, 
\begin{equation}\label{eq-unique-compare-4}
\int_{\R}(u(x,t)-\4{u}_{\mu_1}(x,t))_+\,dx\le 0
\quad\forall 0<t<t_2.
\end{equation}
Since $t_2$ is arbitrary,
\begin{equation*}
u(x,t)\le\4{u}_{\mu_1}(x,t) \qquad \forall x\in\R,\,\,0<t<T_{\mu}.
\end{equation*}
Similarly
\begin{equation*}
u(x,t)\ge\4{u}_{\mu_2}(x,t) \qquad \forall x\in\R,\,\,0<t<T_{\mu_2}.
\end{equation*}
Hence
\begin{equation}\label{eq-compare-infty-1}
\left\{\begin{aligned}
&\frac{\4{u}_{\mu_2}^m(x,t)}{mx}\le\frac{u^m(x,t)}{mx}
\le\frac{\4{u}_{\mu_1}^m(x,t)}{mx}\quad\forall 0<t<T_{\mu_2},x>0\\
&\frac{\4{u}_{\mu_2}^m(x,t)}{mx}\ge\frac{u^m(x,t)}{mx}
\ge\frac{\4{u}_{\mu_1}^m(x,t)}{mx}\quad\forall 0<t<T_{\mu_2}, x<0.
\end{aligned}\right.
\end{equation}
Let $0<a<b<T_{\mu}$ and $\3>0$. We now choose $\mu_2>\mu$ sufficiently close 
to $\mu$ such that $T_{\mu_2}>b$ and $\max\{\mu_2-\mu, \mu-\mu_1\}
<\frac{\epsilon}{2}$. Since $\4{u}_{\mu_1}$, $\4{u}_{\mu_2}$,
satisfies \eqref{eq-condition-neumann-2} with $\mu=\mu_1, \mu_2$, 
there exists $r_1>1$ such that $\forall x\geq r_1$, $a\leq t\leq b$,
\begin{equation}\label{eq-compare-infty-2}
\begin{cases}
\frac{\tilde{u}_{\mu_1}^m(x,t)}{mx}<-\mu_1+\frac{\epsilon}{2}
<-\mu+\epsilon\\
\frac{\tilde{u}_{\mu_2}^m(x,t)}{mx}>-\mu_2-\frac{\epsilon}{2}
>-\mu-\epsilon
\end{cases}
\end{equation}
and $\forall x\leq -r_1$, $a\leq t\leq b$,
\begin{equation}\label{eq-compare-infty-3}
\begin{cases}
\frac{\tilde{u}_{\mu_1}^m(x,t)}{mx}>\mu_1-\frac{\epsilon}{2}>\mu-\epsilon\\
\frac{\tilde{u}_{\mu_2}^m(x,t)}{mx}<\mu_2+\frac{\epsilon}{2}<\mu+\epsilon.
\end{cases}
\end{equation}
By \eqref{eq-compare-infty-1}, \eqref{eq-compare-infty-2},
and \eqref{eq-compare-infty-3}
\begin{equation*}
\left|\frac{u^m(x,t)}{m|x|}+ \mu\right|<\3\quad\forall |x|\geq r_1, 
a\le t\le b.
\end{equation*}
Hence $u$ satisfies \eqref{eq-condition-neumann-2} and the lemma follows.
\end{proof}

\begin{cor}\label{cor-step-initial-4}
Suppose $0\le u_0\in L^{\infty}(\R)$ satisfies \eqref{eq-condition-u_0}
for some constants $\mu_0>0$, $R_0>1$, and 
$f=\sum_{i=1}^{i_0}\mu_i \chi_{I_i}$ is a step function on $[0,T)$ 
where $0=a_0<a_1<\cdots<a_{i_0}=T_0$ is a partition of the interval 
$[0,T_0]$, $I_i=[a_{i-1},a_i)$, $\mu_i\geq 0$ for all $i=1,2,\cdots,i_0$ 
such that 
\begin{equation*}
2\sum_{i=1}^{i_0}\mu_i(a_i-a_{i-1})\geq \int_{\R}u_0\,dx.
\end{equation*}
Let $u$ be the solution of \eqref{eq-Cauchy-problem} in $\R\times(0,T)$ 
given by Lemma \ref{lemma-step-initial-1} which satisfies 
\eqref{main-integral-condition-1} where
$T$ is given by \eqref{eq-def-T-1}. Let $j_0\in\{1,\dots,i_0\}$ be such that 
$a_{j_0-1}<T\le a_{j_0}$. Then $u$ satisfies
\eqref{eq-limit-condition-infty}
uniformly in $[a,b]$ for all $a'_{i-1}<a<b<a_i'$
with $a_i'=a_i$ for all $i=1,2,\cdots,j_0-1$, and $a_{j_0}'=T$.
\end{cor}

By Corollary \ref{cor-step-initial-4}, (3.25), and an argument
similar to the proof of Theorem 1.11 of \cite{Hs2} we have the 
following lemma.

\begin{lem}\label{lem-general-f-exist-1}
Suppose $0\le u_0\in L^{\infty}(\R)$ satisfies \eqref{eq-condition-u_0}
for some constants $\mu_0>0$, $R_0>1$, and 
$0\leq f\in C([0,\infty))$. If $u$ is the solution of 
\eqref{eq-Cauchy-problem} in $\R\times(0,T)$ given by 
Lemma \ref{lemma-step-initial-2} which satisfies 
\eqref{main-integral-condition-1}, then $u$ satisfies
\eqref{eq-limit-condition-infty} uniformly in $[a,b]$ for any $0<a<b<T$.
\end{lem}

\begin{lem}[cf. Theorem 1.12 in \cite{Hs2}]\label{lem-L1-contraction-1}
Suppose $0\leq u_{0,1}\leq u_{0,2}\in L^1(\R)$ and $f_1$, 
$f_2\in C([0,\infty))$ are such that $f_1>f_2$ on $[0,\infty)$. If 
$u_1$, $u_2$ are the solutions of \eqref{eq-Cauchy-problem}
in $\R\times(0,T)$ with initial dates $u_{0,1}$, $u_{0,2}$ which satisfy 
\eqref{main-integral-condition-1} on $(0,T)$ with $u_0=u_{0,1}, u_{0,2}$ 
and $f=f_1, f_2$, respectively, and \eqref{eq-limit-condition-infty} uniformly
on $[a,b]$ with $f=f_1, f_2$, respectively for any $0<a<b<T$, then 
$u_1\leq u_2$ on $\R\times(0,T)$.
\end{lem}
\begin{proof}
Let $\vp\in C_0^{\infty}(\R)$, $0\leq \vp\leq 1$ be such that $\vp(x)=1$ 
for $|x|\leq 1$ and $\vp(x)=0$ for $|x|\geq 2$. For any $R>0$, let 
$\vp_{R}=\vp(x/R)$. Then by the Kato inequality \cite{K},
\begin{equation}\label{eq-partial-t-1}
\begin{aligned}
\frac{\partial}{\partial t}\int_{\R}(u_1-u_2)_+(x,t) \vp_{R}(x)\,dx 
&\leq \int_{\R}\left(\frac{u_1^m}{m}-\frac{u_2^m}{m}\right)_+
\vp_{R,xx}(x)\,dx\\
&\leq \frac{C}{R^2}\int_{R\leq |x|\leq 2R}
\left(\frac{u_1^m}{m}-\frac{u_2^m}{m}\right)_+\,dx \qquad 0<t<T.
\end{aligned}
\end{equation}
Since $f_1>f_2$ on $[0,\infty)$, there exists a constant $\3>0$ such that 
$f_1-f_2>\3$ on $[0,T]$. Let $0<a<b<T$. Since both $u_1$ and $u_2$ satisfy 
\eqref{eq-limit-condition-infty} uniformly on $[a,b]$ with $f=f_1$, $f_2$, 
respectively. There exist a constant $r_0>1$ such that
\begin{equation*}
\begin{cases}
\frac{u_1^m(x,t)}{m|x|}<-f_1(t)+\frac{\3}{2} \qquad \forall |x|
\geq r_0,\,\,a\leq t\leq b\\
\frac{u_2^m(x,t)}{m|x|}>-f_2(t)-\frac{\3}{2} \qquad \forall |x|
\geq R_0,\,\,a\leq t\leq b
\end{cases}
\end{equation*}
Hence
\begin{equation*}
\frac{u_1^m(x,t)}{m}-\frac{u_2^m(x,t)}{m}<(f_2(t)-f_1(t)+\3)|x|<0
\end{equation*}
for all $|x|\geq r_0$, $a\leq t\leq b$. By \eqref{eq-partial-t-1} we get
\begin{equation*}
\frac{\partial}{\partial t}\int_{\R}(u_1-u_2)_+(x,t) \vp_{R}(x)\,dx\leq 0 
\qquad\forall R\geq r_0, a\leq t\leq b.
\end{equation*}
Hence
\begin{align*}
&\int_{\R}(u_1-u_2)_+(x,b)\vp_{R}(x)\,dx 
\leq \int_{\R}(u_1-u_2)_+(x,a)\vp_{R}(x)\,dx\nonumber\\
\Rightarrow\quad
&\int_{\R}(u_1-u_2)_+(x,b)\,dx \leq \int_{\R}(u_1-u_2)_+(x,a)\,dx \qquad 
\mbox{as $R\to \infty$}
\end{align*}
 for all $0<a<b<T$. Letting $a\to 0$, 
\begin{align*}
&\int_{\R}(u_1-u_2)_+(x,b)\,dx \leq 0 \qquad \forall 0<b<T\\
\Rightarrow\quad&u_1\le u_2\quad\mbox{ in }\R\times (0,T)
\end{align*}
and the theorem follows.
\end{proof}

\begin{thm}\label{thm-u_k-u-limit}
Let $0\le u_0\in L^{\infty}(\R)$ satisfy \eqref{eq-condition-u_0}
for some constants $\mu_0>0$, $R_0>1$, $0\le f\in C([0,\infty))$, 
and $T$ be given by \eqref{eq-def-T-1}. Suppose $u$ is the solution of \eqref{eq-Cauchy-problem} 
in $\R\times(0,T)$ which satisfies \eqref{main-integral-condition-1}
and \eqref{eq-limit-condition-infty} uniformly on 
$[a,b]$ for any $0<a<b<T$. Let 
$\{f_k\}_{k+1}^{\infty}\subset C([0,\infty))$ be a sequence of 
functions such that $f_k>f_{k+1}>f\ge 0$ on $[0,T]$ for all 
$k=1,2,\cdots$, and $f_k\to f$ in $L^1([0,T])$ as $k\to\infty$. 
For each $k=1,2,\cdots$, let $u_k$ be a solution of \eqref{eq-Cauchy-problem} in 
$\R\times(0,T_k)$ which satisfies \eqref{main-integral-condition-1}, 
\eqref{eq-def-T-1}, with $f$ and $T$ being replaced by $f_k$ and 
$T_k$ and \eqref{eq-limit-condition-infty} uniformly on $[a,b]$ for 
any $0<a<b<T_k$. Then $u_k$ converges uniformly on every compact 
subset of $\R\times (0,T)$ to $u$ as $k\to\infty$. 
\end{thm}
\begin{proof}
By Lemma \ref{lem-L1-contraction-1},
\begin{equation}\label{eq-u_k-u-increase}
u_{k}(x,t)\leq u_{k+1}(x,t)\leq u(x,t) \qquad \forall (x,t)\in 
\R\times(0,T_k),\,\,k=1,2,\cdots,
\end{equation}
and by \eqref{main-integral-condition-1} $T_k$ increases to $T$ as 
$k\to\infty$. Hence for any $k_0\in \Z^+$ the equation 
\eqref{main-very-fast-diffusion-1} for the sequence 
$\{u_k\}_{k\geq k_0}$ is uniformly parabolic on every compact 
subset of $\R\times(0,T_{k_0})$. Hence by the standard Schauder 
estimates \cite{LSU} the sequence $\{u_k\}_{k\geq k_0}$ are 
equi-H\"older continuous on every compact subset of 
$\R\times(0,T_{k_0})$. By the Ascoli Theorem and a diagonalization
argument the sequence $\{u_k\}_{k=1}^{\infty}$ has a subsequence
$\{u_{k_i}\}_{i=1}^{\infty}$ that converge uniformly to some function 
$v$ on every compact subset of $\R\times(0,T)$ as $i\to\infty$.  
Then by \eqref{eq-u_k-u-increase} the sequence $\{u_k\}_{k=1}^{\infty}$ 
converges uniformly to $v$ on every compact subset of $\R\times(0,T)$ as 
$i\to\infty$. By \eqref{eq-u_k-u-increase},
\begin{equation}\label{eq-comparison-pointwise-1}
v(x,t)\leq u(x,t) \qquad \forall (x,t)\in\R\times(0,T).
\end{equation}
Now since $u_k$ satisfies
\begin{equation*}
\int_{\R}u_k(x,t)\,dx=\int_{\R}u_0\,dx-2\int_0^tf_k\,ds \qquad 
\forall 0\leq t<T_k, 
\end{equation*}
letting $k\to\infty$ we get
\begin{equation*}
\int_{\R}v(x,t)\,dx=\int_{\R}u_0\,dx-2\int_0^tf\,ds \qquad \forall 0\leq t<T. 
\end{equation*}
Since
\begin{equation*}
\int_{\R}u(x,t)\,dx=\int_{\R}u_0\,dx-2\int_0^tf\,ds \qquad \forall 0\leq t<T,
\end{equation*}
we have
\begin{equation}\label{eq-comparison-integral-1}
\int_{\R}u(x,t)\,dx=\int_{\R}v(x,t)\,dx \qquad 0\leq t<T.
\end{equation}
By \eqref{eq-comparison-pointwise-1} and \eqref{eq-comparison-integral-1}, 
$u=v$ on $\R\times(0,T)$ and the theorem follows.
\end{proof}

\begin{thm}\label{thm-unique}
Let $0\le u_0\in L^{\infty}(\R)$ satisfy \eqref{eq-condition-u_0}
for some constants $\mu_0>0$, $R_0>1$, $0\le f\in C([0,\infty))$, 
and $T$ be given by \eqref{eq-def-T-1}. Suppose $u$ is the solution of 
\eqref{eq-Cauchy-problem} 
in $\R\times(0,T)$ which satisfies \eqref{main-integral-condition-1}
and \eqref{eq-limit-condition-infty} uniformly on 
$[a,b]$ for any $0<a<b<T$ and $\4{u}$ is the solution
of \eqref{eq-Cauchy-problem} in $\R\times(0,T)$ constructed in 
\cite{Hu3} by Neumann 
approximation which also satisfies \eqref{main-integral-condition-1} and 
\eqref{eq-limit-condition-infty} uniformly on 
$[a,b]$ for any $0<a<b<T$. Then $u=\4{u}$ in $\R\times (0,T)$ 
\end{thm}
\begin{proof}
We choose a sequence of functions $\{f_k\}_{k+1}^{\infty}\subset 
C([0,\infty))$ such that $f_k>f_{k+1}>f\ge 0$ on $[0,T]$ for all 
$k=1,2,\cdots$, and $f_k\to f$ in $L^1([0,T])$ as $k\to\infty$. 
For each $k=1,2,\cdots$, let $u_k$ be a solution of \eqref{eq-Cauchy-problem} 
in $\R\times(0,T_k)$ which satisfies \eqref{main-integral-condition-1}, 
\eqref{eq-def-T-1}, with $f$ and $T$ being replaced by $f_k$ and 
$T_k$ and \eqref{eq-limit-condition-infty} uniformly on $[a,b]$ for 
any $0<a<b<T_k$. Then by Theorem \ref{thm-u_k-u-limit}
$$
u=\4{u}=\lim_{k\to\infty}u_k
$$
and the theorem follows.
\end{proof}

We are now ready for the proof of Theorem 3.1.
\vspace{6pt}

\noindent{\ni{\it Proof of Theorem 3.1}:}  
Let $\{R_k\}_{k=1}^{\infty}$ be a sequence of positive numbers such that
$R_k\to\infty$ as $k\to\infty$. By Theorem 3.2 the sequence 
$\{R_k\}_{k=1}^{\infty}$ has a subsequence which we may assume without 
loss of generality to be the sequence itself such that $u^{R_k}$ converges
uniformly on every compact subset of $\R\times (0,T)$ as $k\to\infty$ to a
solution $u$ of \eqref{eq-Cauchy-problem} which satisfies 
\eqref{eq-thm-condition-1} where $T$ is given by \eqref{eq-mu-time}. 
By Lemma \ref{lemma-step-initial-3} $u$ satisfies 
\eqref{eq-condition-neumann-2} uniformly on $[a,b]$ for any $0<a<b<T$. 
By Theorem 4.8 $u$ is independent of the choice of sequence 
$\{R_k\}_{k=1}^{\infty}$. Hence $u^R$ converges
uniformly on every compact subset of $\R\times (0,T)$ to $u$ as 
$R\to\infty$ and the theorem follows.
 
\hfill$\square$\vspace{6pt}

\begin{thm}\label{thm-convergence-Dirichlet-problem-2}
Let $0\le u_0\in L^{\infty}(\R)$ satisfy \eqref{eq-condition-u_0}
for some constants $\mu_0>0$, $R_0>1$, $f,\,g \in C([0,\infty))$ be such that
$f(t),\, g(t)\ge\mu_0$ on $[0,\infty)$, 
and $T$ be given by \eqref{eq-def-T-general}.
Let $v^R$ be the solution
of \eqref{problem-general-Dirichlet}.
Then $v^R$ converges uniformly on every compact subset of $\R\times 
(0,T)$ to a solution $u$ of \eqref{eq-Cauchy-problem} which satisfies 
\eqref{main-integral-condition-general}, \eqref{eq-limit-condition-infty+general} and \eqref{eq-limit-condition-infty-general} uniformly on $[a,b]$ for any $0<a<b<T$ as $R\to\infty$. Moreover, the solution is the same as the solution of \eqref{eq-Cauchy-problem} in $\R\times(0,T)$ constructed in \cite{Hu3} by Neumann approximation method.
\end{thm}
\begin{proof}
Let $\{R_k\}_{k=1}^{\infty}$ be a sequence of positive numbers such that $R_k\to \infty$ as $k\to\infty$ and let $v_{\epsilon}^{R_k}$ be the solution of \eqref{problem-general-Dirichlet} with initial data $v_{\epsilon}^{R_k}(x,0)=u_0(x)+\epsilon$. Let $\mu=\max\left(\|f\|_{L^{\infty}(0,T)}, \|g\|_{L^{\infty}(0,T)}\right)$ and let $u_{\epsilon}^{R_k,\mu}$ be the solution of \eqref{problem-Dirichlet} with initial data $u_{\epsilon}^{R_k,\mu}(x,0)=u_0(x)+\epsilon$. Then, by maximum principle, we have
\begin{equation*}
u^{R_k,\mu}_{\epsilon}\leq v^{R_k}_{\epsilon}
\end{equation*}
\begin{equation*}
\Rightarrow u^{R_k,\mu}\leq v^{R_k} \qquad \mbox{in $\R\times(0,\infty)$}, \qquad \mbox{as $\epsilon \to 0$}.
\end{equation*}
Let 
\begin{equation}\label{eq-maximum-extinction-time-T-1}
T_0=\frac{1}{2\mu}\int_{\R}u_0\,dx.
\end{equation}
By Theorem \ref{thm-convergence-Dirichlet-problem}, $u^{R_k,\mu}$ converges uniformly on any compact subsets of $\R\times(0,T_0)$ as $\R_k\to\infty$ to a solution $\tilde{u}$ of \eqref{eq-Cauchy-problem} which satisfies 
\eqref{eq-thm-condition-1} and \eqref{eq-condition-neumann-2} 
uniformly on $[a,b]$ for any $0<a<b<T_0$ in $\R\times(0,T_0)$. Let $K_1$ be a compact subset of $\R\times(0,T_0)$. Then there exist a constant $c_0=c_0(K_1)>0$ such that  
\begin{equation*}
\tilde{u}\geq c_0>0 \qquad \mbox{on $K_1$}.
\end{equation*}
Hence there exists a constants $k_0\in \Z^+$ and $C(K_1)>0$ such that 
\begin{equation*}
v^{R_k}\geq u^{R_k, \mu}\geq C(K_1)>0 \qquad k\geq k_0>>1.
\end{equation*}
Thus the sequence $\{v^{R_k}\}_{k=1}^{\infty}$ is uniformly bounded below by some positive constant on any compact subset of $\R\times(0,T_0)$ for all $k$ sufficiently large. Since the sequence $\{v^{R_k}\}_{k=1}^{\infty}$ is uniformly bounded from above by $\|u_0\|_{L^{\infty}}$, the equation \eqref{main-very-fast-diffusion-1} for the sequence $\{v^{R_k}\}_{k=1}^{\infty}$ is uniformly parabolic on every compact subset $\R\times(0,T_0)$. Hence by the Schauder estimates for parabolic equations 
\cite{LSU}, the sequence $\{v^{R_k}\}_{k=1}^{\infty}$ is equi-H\"older continuous on 
every compact subsets of $\R\times(0,T_0)$. Hence any sequence
$\{v^{R_k}\}_{k=1}^{\infty}$ with $R_k\to\infty$ as $k\to\infty$
has a subsequence which we may assume without loss of generality 
to be the sequence itself that converges uniformly on 
every compact subset of $\R\times (0,T_0)$ to a solution $v$ of \eqref{main-very-fast-diffusion-1} in $\R\times(0,T_0)$ as 
$k\to\infty$. Since by \eqref{eq-condition-u_0}, $v^{R_k}$ satisfies \eqref{eq-u^R-bound} for $R_k\geq R_0$. By an argument similar to the proof of Theorem \ref{thm-sequential-convergence-problem}, $v$ has initial value $u_0$. Hence $v$ is a solution of \eqref{eq-Cauchy-problem} in $\R\times(0,T_0)$.

\indent It remains to show that $v$ satisfies \eqref{main-integral-condition-general}. For any $j=1,2,\cdots,$ let 
$f_j=\sum_{i=1}^{2^j}\mu_i\chi_{I_i}$, $g_j=\sum_{i=1}^{2^j}\nu_i\chi_{I_i}$ where $\mu_i=\sup_{I_i}f+\frac{1}{j}$, $\nu_i=\sup_{I_i}g+\frac{1}{j}$,
$I_i=[a_{i-1},a_i)$, $a_0=0$, $a_i=iT/2^j$, for all $i=1,2,\cdots,2^j$.\\
\indent We now consider the solution $v_j^{R_k}(x,t)$ of following Neumann problem
\begin{equation*}
\begin{cases}
v_t=\left(\frac{v^m}{m}\right)_{xx} \qquad \qquad I_{R_k}\times(0,\infty)\\
\left(\frac{v^m}{m}\right)_x(R_k,t)=-f_j \qquad \forall 0<t<T_{j,k}\\
\left(\frac{v^m}{m}\right)_x(-R_k,t)=g_j \qquad \forall 0<t<T_{j,k}\\
v(x,0)=u_0(x) \qquad \qquad \mbox{in $I_{R_k}$}
\end{cases}
\end{equation*}
which satisfies 
\begin{equation*}
\frac{v_t}{v}\leq \frac{1}{(1-m)t} \qquad \mbox{in $I_R\times(0,T_{j,k})$}
\end{equation*}
and 
\begin{equation*}
\int_{-R_k}^{R_k}v(x,t)\,dx=\int_{-R_k}^{R_k}u_0\,dx-\int_0^t(f_j+g_j)\,ds \qquad \forall 0\leq t<T_{j,k}
\end{equation*}
where $T_{j,k}$ is given by
\begin{equation*}
\int_{-R_k}^{R_k}u_0\,dx=\int_0^{T_{j,k}}(f_j+g_j)\,ds.
\end{equation*}
Then, by Lemma 4.2 of \cite{Hu3}, the solution $v_j^{R_k}(x,t)$ has a subsequence which we may assume without loss of generality to be the sequence itself that converges to the solution $v_j(x,t)$ of \eqref{eq-Cauchy-problem} uniformly on every compact subset of $\R\times(0,T_0)$ as $k\to\infty$ with
\begin{equation*}
\int_{\R}v_j(x,t)\,dx=\int_{\R}u_0\,dx-\int_0^t(f_j+g_j)\,ds \qquad \forall 0\leq t<T_{j}
\end{equation*}
where
\begin{equation*}
\int_{\R}u_0\,dx=\int_0^{T_{j}}(f_j+g_j)\,ds.
\end{equation*}
Let $i_j\in\{1,2,\cdots,2^j\}$ such that $a_{i_j-1}<T_{j}\leq a_{i_j}$. Then, by \cite{Hu3}, the solution $v_j$ also satisfies, for all $\epsilon>0$, 
\begin{equation*}
\frac{v_j^m(x,t)}{mx} \to -\mu_i \quad \mbox{uniformly on $[a_{i-1}+\epsilon, a_i-\epsilon]$}\quad \mbox{as $x\to \infty$},\quad \forall i=1,\cdots,i_{j}-1
\end{equation*}
and 
\begin{equation*}
\frac{v_j^m(x,t)}{mx} \to \nu_i \quad \mbox{uniformly on $[a_{i-1}+\epsilon, a_i-\epsilon]$}\quad \mbox{as $x\to -\infty$},\quad \forall i=1,\cdots,i_{j}-1.
\end{equation*}
Hence, for sufficiently large $R_k>>1$,
\begin{equation*}
\begin{cases}
\frac{v_j^m(R_k,t)}{mR_k}<-\sup_{I_i}f-\frac{1}{2j} \qquad \forall t\in[a_{i-1}+\epsilon, a_i-\epsilon], \quad \forall i=1,\cdots,i_{j}-1\\
\frac{v_j^m(-R_k,t)}{m(-R_k)}>\sup_{I_i}g+\frac{1}{2j} \qquad \forall t\in[a_{i-1}+\epsilon, a_i-\epsilon], \quad \forall i=1,\cdots,i_{j}-1.
\end{cases}
\end{equation*}
and
\begin{equation}\label{eq-case-pm-R-k-lim-1}
\begin{cases}
v_j(R_k,t)<\left(|m|R_k f(t)\right)^{\frac{1}{m}}=v^{R_k}(R_k,t)\\ 
v_j(-R_k,t)<\left(|m|R_k g(t)\right)^{\frac{1}{m}}=v^{R_k}(-R_k,t)
\end{cases}
\end{equation}
for any $t\in[a_{i-1}+\epsilon, a_i-\epsilon]$, $\forall i=1,\cdots,i_{j}-1$. Hence by \eqref{eq-case-pm-R-k-lim-1} and an argument similar to the proof of Lemma 2.3 of \cite{DK} and Lemma 2.5 of \cite{Hu3}, for sufficiently large $R_k>>1$,
\begin{equation}\label{eq-comparison-for-lowerbound-1}
\int_{|x|\leq R_k}(v_j-v^{R_k})_+(x,t_2)\,dx \leq \int_{|x|\leq R_k}(v_j-v^{R_k})_+(x,t_1)\,dx
\end{equation} 
for $ a_{i-1}+\epsilon\leq t_1 \leq t_2 \leq a_{i}-\epsilon, \,\, i=1,2,\cdots,i_j-1$. \\
Letting $k\to\infty$ in \eqref{eq-comparison-for-lowerbound-1}, by \eqref{eq-u^R-bound} and Lebesque Dominated Convergence Theorem,
\begin{equation*}
\int_{\R}(v_j-v)_+(x,t_2)\,dx \leq \int_{\R}(v_j-v)_+(x,t_1)\,dx 
\end{equation*} 
for any $a_{i-1}+\epsilon \leq t_1 \leq t_2 \leq a_i-\epsilon$ and $t_2< T_0$. For $i=1,\cdots,i_j-1$, letting $\epsilon \to 0$,  $t_1\to a_{i-1}$ and taking $t_2$ arbitrary,
\begin{equation}\label{eq-comparison-DK-1}
\int_{\R}(v_j-v)_+(x,t)\,dx \leq \int_{\R}(v_j-v)_+(x,a_{i-1})\,dx \qquad \forall a_{i-1}\leq t\leq a_i,\,\,i=1,\cdots, i_{j}-1.
\end{equation}
Similarly
\begin{equation}\label{eq-comparison-DK-2}
\int_{\R}(v_j-v)_+(x,t)\,dx \leq \int_{\R}(v_j-v)_+(x,a_{i-1})\,dx \qquad \forall a_{i_j-1}\leq t\leq T_{j}.
\end{equation}
Hence by \eqref{eq-comparison-DK-1} and \eqref{eq-comparison-DK-2}, 
\begin{equation*}
\int_{\R}(v_j-v)_+(x,t)\,dx \leq \int_{\R}(v_j-v)_+(x,0)\,dx=0
\end{equation*}
for $0<t<T_{j}$. $\forall j\in \Z^+$. Then
\begin{equation*}
v_j\leq v \qquad \mbox{in $\R\times(0,T_j')$} \quad \forall j=1,2,\cdots,
\end{equation*}
where $T_j'=\min\left(T_0,T_{j}\right)$. Therefore
\begin{equation*}
\int_{\R}v(x,t)\,dx \geq \int_{\R}v_j(x,t)\,dx=\int_{\R}u_0(x)\,dx-\int_{0}^{t}(f_j+g_j)\,ds \qquad \forall 0\leq t <T_j'.
\end{equation*}
Letting $j\to\infty$, we have
\begin{equation*}
\int_{\R}v(x,t)\,dx \geq \int_{\R}u_0(x)\,dx-\int_{0}^{t}(f+g)\,ds \qquad \forall 0\leq t< T_0.
\end{equation*}
Similarly, one can prove that
\begin{equation*}
\int_{\R}v(x,t)\,dx \leq \int_{\R}u_0(x)\,dx-\int_{0}^{t}(f+g)\,ds \qquad \forall 0\leq t< T_0.
\end{equation*}
Hence $v$ satisfies \eqref{main-integral-condition-general} for any $t\in[0,T_0]$.\\
\indent Let $\tilde{T}_0\geq T_0$ be the maximal time such that $\{v^{R_k}\}_{k=1}^{\infty}$ has a subsequence which we still denote by $\{v^{R_k}\}_{k=1}^{\infty}$ that converges to a solution $v$ of \eqref{eq-Cauchy-problem} in $\R\times(0,\tilde{T}_0)$ which satisfies \eqref{main-integral-condition-general} for $0\leq t<\tilde{T}_0$ as $k\to\infty$. We claim that $\tilde{T}_0=T$. Suppose not. Then $\tilde{T}_0<T$. Hence by \eqref{main-integral-condition-general}, 
\begin{equation}\label{eq-expand-extinction-time-1}
\int_{\R}v(x,\tilde{T}_0)\,dx=\int_{\R}u_0\,dx-\int_0^{\tilde{T}_0}(f+g)\,ds>0.
\end{equation}
We will now choose a constant $T_0'<\tilde{T}_0$ sufficiently close to $\tilde{T}_0$. Let $u^{R,\mu}_1$ be the solution of \eqref{problem-Dirichlet} with initial value $v(x,T_0')$. By Theorem \ref{thm-convergence-Dirichlet-problem}, $u_1^{R,\mu}$ converges uniformly on any compact subsets of $\R\times(0,\tilde{T})$ as $k\to\infty$ to the solution $\tilde{u}$ of \eqref{eq-Cauchy-problem} with $u_0(x)=v(x,T_0')$ where
\begin{equation*}
\tilde{T}=\frac{1}{2\mu}\int_{\R}v(x,T_0')\,dx.
\end{equation*}
Then by repeating the previous argument using $u_1^{R_k,\mu}$ as the comparison function, we get that $v^{R_k}(x,t+T_0')$ has a subsequence which we still denote by $v^{R_k}$ such that $v^{R_k}(x,t+T_0')$ converges to a solution $\tilde{v}$ of \eqref{eq-Cauchy-problem} in $\R\times(0,\tilde{T})$ with $u_0(x)=v(x,T_0')$ where
\begin{equation*}
\begin{aligned}
\tilde{T}=\frac{1}{2\mu}\int_{\R}v(x,T_0')\,dx&=\frac{1}{2\mu}\left(\int_{\R}u_0\,dx-\int_0^{T_0'}(f+g)\,ds\right)\\ 
&\geq \frac{1}{2\mu}\left(\int_{\R}u_0\,dx-\int_0^{\tilde{T}_0}(f+g)\,ds\right):=C_1>0.
\end{aligned}
\end{equation*}
We extend $v$ to a solution of \eqref{eq-Cauchy-problem} in $\R\times(0,T_0'+\tilde{T})$ by setting $v(x,t)=\tilde{v}(x,t-T_0')$ for $T_0'\leq t <T_0'+\tilde{T}$. We now choose $T_0'>0$ such that $T_0-\frac{C_1}{2}<T_0'<\tilde{T}_0$. Then
\begin{equation*}
T_0'+\tilde{T}>T_0.
\end{equation*}
This contradicts the maximality of $\tilde{T}_0$. Therefore
\begin{equation*}
\tilde{T}_0=T.
\end{equation*}
Hence $\{v^{R_k}\}_{k=1}^{\infty}$ has a subsequence which we still denote by $\{v^{R_k}\}_{k=1}^{\infty}$ such that $v^{R_k}$ converges to a solution $v$ of \eqref{eq-Cauchy-problem} in $\R\times(0,T)$ which satisfies \eqref{main-integral-condition-general} for $t\in (0,T)$ as $k\to\infty$.\\
\indent By an argument similar to the proof of Corollary \ref{cor-step-initial-4}, Lemma \ref{lem-general-f-exist-1} and the proof of Theorem 1.11 of \cite{Hu2}, $u$ satisfies \eqref{eq-limit-condition-infty+general} and \eqref{eq-limit-condition-infty-general} for any $0<a<b<T$. Then by \eqref{eq-limit-condition-infty+general}, \eqref{eq-limit-condition-infty-general} and the same argument as the proof of Theorem \ref{thm-unique}, $u$ is equal to the solution $\tilde{u}$ of \eqref{eq-Cauchy-problem} in $\R\times(0,T)$ constructed in \cite{Hu3} by Neumann approximation method.\\
\indent Since the sequence $\{v^{R_k}\}_{k=1}^{\infty}$ is arbitrary and the limit of the sequence $u=\tilde{u}$ is unique and independent of the sequence $\{R_k\}_{k=1}^{\infty}$, $v^{R}$ converges uniformly to $u$ every compact subset of $\R\times(0,T)$ as $R\to\infty$ and the theorem follows.
\end{proof}


\begin{thebibliography}{99}

\bibitem[A]{A} D.G.~Aronson, {\em The porous medium equation, CIME
Lectures}, in Some problems in Nonlinear Diffusion, Lecture
Notes in Mathematics 1224, Springer-Verlag, New York, 1986.

\bibitem[DDD]{DDD} S.H.~Davis, E.~Dibenedetto, and D.J.~Diller,
{\em Some a priori estimates for a singular evolution equation
arising in thin-film dynamics}, SIAM J. Math. Anal. 27 (1996),
no. 3, 638--660.

\bibitem[DD]{DD} E.~Dibenedetto and D.J.~Diller, {\em About a
singular parabolic equation arising in thin film dynamics and
in the Ricci flow for complete $\Bbb{R}^2$}, Partial differential
equations and applications, 103--119, Lecture Notes in
Pure and Applied Mathematics, Vol. 177  , edited by P.~Marcellini,
Giorgio G.Talenti and E.~Vesentini, Dekker, New York, 1996.

\bibitem[DK]{DK} B.E.J.~Dahlberg and C.~Kenig, {\em Non-negative
solutions of generalized porous medium equations},
Revista Matem\'atica Iberoamericana 2 (1986), 267--305.

\bibitem[DP]{DP} P.~Daskalopoulos and M.A.Del Pino, {\em On a singular 
diffusion equation,} Comm. in Analysis and Geometry 3 (1995), no. 3,
523-542.

\bibitem[Hs1]{Hs1} S.Y.~Hsu, {\em Large time behaviour of solutions of
the Ricci flow equation on $R^2$}, Pacific J. Math. 197 (2001), no. 1,
25--41.

\bibitem[Hs2]{Hs2} S.Y.~Hsu, {\em Uniqueness of Solutions of a Singular 
Diffusion Equation,} Differential and Integral Equations 16 (2003), 
no. 2, 181-200.

\bibitem[Hs3]{Hs3} S.Y.~Hsu, {\em Classification of Radially Symmetric 
Self-similar Solutions of $u_t=\La \log u$ in Higher Dimensions,} 
Differential and Integral Equations 18 (2005), no. 10, 1175-1192.

\bibitem[Hu1]{Hu1} K.M.~Hui, {\em Existence of solutions of the 
equation $u_t=\Delta\log u$}, Nonlinear Anal. TMA,  37  (1999),  no. 7, 
875--914. 

\bibitem[Hu2]{Hu2} K.M.~Hui, {\em On some Dirichlet and Cauchy problems 
for a singular diffusion equation.} Differential and Integral Equations, 
15 (2002), 769-804.

\bibitem[Hu3]{Hu3} K.M.~Hui, {\em Existence of solutions of the very fast 
diffusion equation in bounded and unbounded domain.} Math. Ann. 339 
(2007), 395-443.

\bibitem[Hu4]{Hu4} K.M.~Hui, {\em Singular limit of solutions of the very 
fast diffusion equation.} Nonlinear Anal. 68 (2008), no. 5, 1120--1147.

\bibitem[K]{K} T.~Kato, {\em Schr\"odinger operators with
singular potentials}, Israel J. Math. 13 (1973), 135--148.

\bibitem[Ku]{Ku} T.G.~Kurtz, {\em Convergence of sequences of
semigroups of nonlinear operators with an application to gas kinetics},
Trans. Amer. Math. Soc. 186 (1973), 259--272.

\bibitem[LSU]{LSU} O.A.~Ladyzenskaya, V.A.~Solonnikov and N.N.~Uraltceva 
{\em Linear and Quasilinear Equations of Parabolic Type,} Transl. Math. 
Mono., 23, Amer. Math. Soc., Providence, R.I., 1968.

\bibitem[LT]{LT} P.L.~Lions and G.~Toscani, {\em Diffusive limit for
finite velocity Boltzmann kinetic models}, Revista
Matematica Iberoamericana 13 (1997), no. 3, 473--513.

\bibitem[P]{P} L.A.~Peletier, {\em The porous medium equation} in
Applications of Nonlinear Analysis in the Physical Sciences,
H.~Amann, N.~Bazley, K.~Kirchgassner editors, Pitman, Boston, 1981.

\bibitem[ERV]{ERV} A.~Rodriguez, J.R.Esteban and J.L.~Vazquez, 
{\em A nonlinear heat equation with singular diffusivity}, Arch. 
Rational Mech. Analy. 103 (1988), 985--1039.

\bibitem[RV]{RV} A.~Rodriguez and J.L.~Vazquez, {\em A well posed problem 
in singular Fickian diffusion}, Arch. Rational Mech. Analy. 110 (1990),
141--163.

\bibitem[R]{R} G.~Rosen, {\em Nonlinear heat conduction in solid
$H_2$}, Physical Review B 19 (1979), 2398--2399.

\bibitem[V1]{V1} J.L.~Vazquez, {\em Nonexistence of solutions for nonlinear
heat equations of fast-diffusion type}, J. Math. Pures Appl. 71
(1992), 503--526.

\bibitem[V2]{V2} J.L.~Vazquez, {\em The porous medium equation Mathematical
Theory}, Oxford University Press Inc., New York, 2007.

\bibitem[W1]{W1} L.~F.~Wu, {\em The Ricci flow on complete $R^2$},
Comm. in Analysis and Geometry 1 (1993), 439--472.

\bibitem[W2]{W2} L.~F.~Wu, {\em A new result for the porous
medium equation}, Bull. Amer. Math. Soc. 28 (1993), 90--94.


\end{thebibliography}
\end{document}